\documentclass[11pt,reqno]{amsart} 
\usepackage[utf8]{inputenc}
\usepackage{amscd,amsfonts,amsmath,amssymb,amsthm,bbm,dsfont,centernot,mathtools}
\usepackage{color,mathrsfs,stackrel}
\usepackage[left=3.1cm,right=3.1cm,top=3cm,bottom=3cm]{geometry} 
\usepackage[colorlinks=true,linkcolor=blue,citecolor=red]{hyperref}
\usepackage{yhmath}
\usepackage{fancyhdr}
\usepackage{xy}
\usepackage{eso-pic}
\usepackage{xcolor}
\usepackage{extarrows}
\usepackage{relsize}
\usepackage{caption}
\usepackage{eurosym}
\usepackage{chngcntr}
\usepackage{enumitem}
\usepackage{comment}

\newtheorem{corollary}{Corollary}[section]
\newtheorem{lemma}[corollary]{Lemma}
\newtheorem{proposition}[corollary]{Proposition}
\newtheorem{theorem}[corollary]{Theorem}

\theoremstyle{definition}
\newtheorem{definition}[corollary]{Definition}

\newtheorem*{acknowledgements}{\sc Acknowledgements}

\numberwithin{equation}{section}

\allowdisplaybreaks

\def\XXint#1#2#3{{\setbox0=\hbox{$#1{#2#3}{\int}$ }
\vcenter{\hbox{$#2#3$ }}\kern-.6\wd0}}

\def\mint{\Xint{\rotatebox[origin][30]{$-$}}}

\def \div {\mathop {\rm div}\nolimits}

\def \de {\mathrm{d}}

\def \R {\mathbb R}

\def \Sf {\mathbb S}
\def \Z {\mathbb Z}
\def \N {\mathbb N}
\def \O {\Omega }
\def \C {\mathbb C}

\def \W {W^{2,p}(Q)}
\def \Ws {W^{2,p}_{\#}(Q)}
\def \Gh {\Gamma_h}
\def \Ghh {\Gamma^{\#}_h}
\def \Oh {\Omega_h}
\def \Ohh {\Omega^{\#}_h}
\def \Ad {AD_h}
\def \H {\mathcal{H}}
\def \F {\mathcal{F}}

\def \f {f_{i,N}}
\def \g {g_{i,N}}
\def \hi {h_{i,N}}
\def \hk {h_{k,N}}
\def \him {h_{i-1,N}}
\def \hkm {h_{k-1,N}}
\def \ui {u_{i,N}}
\def \uim {u_{i-1,N}}
\def \uk {u_{k,N}}
\def \ukm {u_{k-1,N}}
\def \Ji {J_{i,N}}
\def \Jim {J_{i-1,N}}
\def \Hi {H_{i,N}}
\def \Ht {\widetilde{H}_{N_k}}
\def \Htn {\widetilde{H}_{N}}
\def \Wt {\tilde{W}_{N_k}}
\def \hti {\tilde{h}_{N_k}}
\def \htin {\tilde{h}_{N}}
\def \utn {\tilde{u}_{N}}
\def \Jt {\tilde{J}_{N_k}}
\def \Bi {B_{i,N}}
\def \Gin {\Gamma_{i,N}}
\def \Gi {G_{i,N}}
\def \Gk {G_{k,N}}
\def \Pi {P_{i,N}}
\def \LiQ {L^{\infty}(Q)}
\def \LpQ {L^{p}(Q)}
\def \LdQ {L^{2}(Q)}
\def \LqQ {L^{q}(Q)}

\newcommand{\PPP}{\color{black}} 
\newcommand{\OOO}{\color{black}} 

\title[Evolution of crystalline thin films by evaporation and condensation]{Evolution of crystalline thin films by evaporation and condensation in three dimensions }
\author[P. Piovano] {Paolo Piovano} 
\address[Paolo Piovano]{Dipartimento di Matematica, Politecnico di Milano, P.zza Leonardo da Vinci 32, 20133 Milano, Italy\footnote{MUR Excellence Department 2023-2027}, \& WPI c/o Research Platform MMM ``Mathematics-Magnetism-Materials'', Fak.\ Mathematik Univ.\ Wien, A1090 Vienna}
\email[P. Piovano]{paolo.piovano@polimi.it}
\author[F. Sapio]{Francesco Sapio}
\address[Francesco Sapio]{Wolfgang Pauli Institute, Oskar-Morgenstern Platz 1, 1090 Wien, Austria}
\email{sapiof23@univie.ac.at}


\keywords{Thin films, evaporation, condensation, mismatch strain, 3 dimensions, gradient flow, curvature regularization, minimizing movements, existence, regularity.
\\
\indent \emph{Data sharing not applicable to this article as no datasets were generated or analysed during the current study.}
}

\date{\today}

\begin{document}

\begin{abstract}
The morphology of crystalline thin films evolving on flat rigid substrates by condensation of extra film atoms or by evaporation of their own atoms in the surrounding vapor is studied in the framework of the theory of Stress Driven Rearrangement Instabilities (SDRI). By following \PPP the literature \OOO both the elastic contributions due to the mismatch between the film and the substrate lattices at their theoretical  (free-standing) elastic equilibrium, and 
 a curvature perturbative regularization preventing the problem to be ill-posed due to the otherwise exhibited backward parabolicity,  are added in the evolution equation. The resulting Cauchy problem under investigation consists in an anisotropic mean-curvature type flow of the fourth order on the film profiles, which are  assumed to be parametrizable as graphs of functions measuring the film thicknesses, 
 coupled with a quasistatic elastic problem in the film bulks. \PPP  The existence of a regular solution for a finite period of time  is established under periodic boundary conditions by means of employing minimizing movements to exploit the gradient-flow structure of the evolution equation. \OOO 


\end{abstract}

\maketitle


\section{Introduction}
 In this paper we study the morphological evolution of crystalline thin films deposited on flat rigid substrates by \emph{condensation} of extra film atoms from a surrounding vapor, which results in a film growth, or by \emph{evaporation} in such vapor of their own film atoms, which  triggers instead a film dissolution. 
 Besides on the phenomena at the film surface,  the focus is  on the elastic properties of the film bulk material, which is subject to the so-called \emph{epitaxial strain}  imparted by the underlying substrate. \PPP Such a strain is a consequence of the literature modeling assumptions \cite{FG:2004,D:2001,spencer1997equilibrium,Srolovitz:1989}  under which, \OOO on the one hand, the film and the substrate material are prescribed to adhere at their contact interface without delamination or debonding, \PPP and,  on the other hand,  the crystalline  lattices characterizing the substrate and the film free-standing elastic equilibrium are allowed to present a mismatch of their lattice parameters. \OOO

  \PPP The main result of the paper is the existence of a regular quasistatic solution, namely of a solution presenting a regular profile and being at elastic equilibrium in the bulk at each time (see Theorem \ref{teoconvfin}) for a finite period of time. Therefore,  \OOO the  investigations carried out in this paper can be seen both as a generalization in the three dimensional setting, which is the physically relevant one for the applications, of the  results previously achieved in \cite{PP}, and as an equivalent of the \PPP evolution existence result achieved in \OOO  \cite{FFLM3} (see also \cite{FFLM2-1}) for the complementary setting in which the film evolution is not influenced  by the evaporation-condensation process here considered, but it is entirely due to the volume preserving \emph{surface-diffusion process}, which is instead here neglected. We notice that\PPP, as a crucial difference with respect to \cite{FFLM3}, \OOO in our setting the film volume can suddenly change directly interfering with the elastic properties of the deposited film material \PPP and affecting the possibility of developing a direct counterpart  of the stability analysis carried out in  \cite{FFLM3} for configurations characterized by a flat profile (with a nonzero  film volume). \OOO 
  
 \PPP Furthermore, we remind that the extension of the existence result contained in this paper  to every time is actually not expected to hold for solutions with such regular profiles in the general setting here tackled, which takes into account both elasticity and  the possibility of highly  anisotropic non-convex surface tensions for the film-vapor interface. In fact, as recovered in the static case in \cite{BCh:2002,CF:2020_arma,DP:2018_2,FFLM2,KrP}  for related models in agreement with direct observations, various surface instabilities and bulk cracks may develop as a further strain relief mechanism apart from the bulk material deformation.  Moreover, in \cite{DP:2018_1} it is also shown for related models that, without restricting to surface tensions in the \emph{wetting regime},  film profiles touch the substrate surface in a nonzero Young-Dupr\'e angle, which corresponds to an only Lipschitz regularity for the film profile, also in the presence of elasticity  when the material of the film is less rigid of the one of the substrate. \OOO


\PPP The \OOO reference framework for \PPP such static models \OOO is provided in the literature by the theory of \emph{Stress Driven Rearrangement Instabilities} (SDRI) \cite{AT:1972,D:2001,FG:2004,Gr:1993,Srolovitz:1989} of which  thin films have  historically been the SDRI foremost example, \PPP exactly \OOO because of the interplay occurring \PPP among the various film strain-relief modes with respect to \OOO the film surface and \PPP the \OOO elastic energy: 
 \PPP The \OOO bulk deformation is energetically neutral with respect to the surface energy, but is payed in terms of the elastic energy, while the destabilization of the film free-boundary from the \emph{Winterbottom optimal shape} \cite{PV2, PV, Winterbottom} with surface roughness and interface instabilities positively contributes to the surface energy without storage of elastic energy. A delicate compromise between these competing stress-relief mechanisms must therefore be achieved, of which the film morphology is the result. 
\PPP Notice that a \OOO microscopic justification of  \PPP such \OOO static models starting from atomic interactions has been provided  in \cite{KrP}, while existence and regularity results for dimension $d=2$ are available in \cite{BCh:2002,DP:2018_1,DP:2018_2, FFLM:2011,FFLM2, FM2,KP1, KP2} and have been (partially) extended in higher dimensions for models related to the more general setting of SDRI in \cite{CF:2020_arma,KP3}.

In regard of the evolution theory, the reference goes back to W. W. Mullins \cite{Mul-1,Mul-2}, who identify the equations describing the motion of a crystalline interface  $\Gamma\subset\R^d$ by the evaporation-condensation process and by surface diffusion as the \emph{motion by mean curvature}, i.e.,
\begin{equation}\label{1}
    V=- H \qquad \text{on $\Gamma$,}
\end{equation}
and the \emph{motion by the Laplacian of the mean curvature}, i.e.,
 \begin{equation}\label{2}
V=\Delta_{\Gamma}H \qquad \text{on $\Gamma$,}
\end{equation}
respectively, where $V$ denotes the normal velocity on $\Gamma$, $H$ is the mean curvature on $\Gamma$, and $\Delta_{\Gamma}$ represents the tangential Laplacian along $\Gamma$. 
As described in  \cite{PP} by taking into account the film-vapor interface anisotropy $\psi$  and the elastic contribution, Equations \eqref{1} and  \eqref{2} become
\begin{equation}\label{1bis}
  V=-\mathrm{div}_{\Gamma}(D\psi(\nu))-W(Eu)\qquad \text{on $\Gamma$},
\end{equation}
and
 \begin{equation}\label{2bis}
V = \Delta_{\Gamma}[\mathrm{div}_{\Gamma}(D\psi(\nu))+W(Eu)] \qquad \text{on $\Gamma$},
\end{equation}  
(see  \cite{AG, FG:2004,GS,GSS,H,H2} and \cite[Remark 3.1, Section 8]{G} for more details), where $u(\cdot,t)$ is the elastic equilibrium in the  region $\Omega\subset\R^d$, $Eu$ represents the \emph{strain} and is the symmetric part of the gradient $Du$, $W$ is the elastic energy density, which is defined as
    \begin{equation*}
        W(A):=\frac{1}{2}\C A:A
    \end{equation*}
   for every $A\in \R^{2\times 2}$ and for  a symmetric positive fourth-order tensor  $\C$, $\mathrm{div}_{\Gamma}$ is the tangential divergence along $\Gamma\subset \partial \Omega$, and $\nu$ is the outward unit normal of $\partial\Omega$.   

We notice that \eqref{2bis} has been tackled   in  \cite{FJM,FJM2} by means of fixed point techniques to prove existence and uniqueness results. More precisely,  in  \cite{FJM}  the authors study  \eqref{2bis} for  $d=2$ and obtain short time existence and uniqueness, establishing the global-in-time  existence for a specific class of initial data, while in \cite{FJM2} they prove  short-time existence of a smooth solution and uniqueness for  $d=3$ (with the bulk contribution given by a forcing term, which also include the elastic setting). Furthermore,  in the absence of elasticity the strategy based on \emph{minimizing movements} has been extensively used previously in the literature  to treat geometric flows of the type \eqref{1} and \eqref{2}, also in order to establish  global-in-time existence.  Regarding  the mean curvature flow, without aiming at a comprehensive description of the results by minimizing movements we just mention here that \PPP  in  \cite{CMNP,CMP} \OOO the authors obtain global-in-time existence and uniqueness for anisotropic \PPP and crystalline mean curvature flows in all dimensions also \OOO for arbitrary (possibly unbounded) initial sets and for the natural mobility, while \PPP in \cite{J,MSS} \OOO global-in-time existence of the flat solutions for a mean curvature flow with volume preserving is established.


Regarding the literature results for the Equations  \eqref{1bis} and \eqref{2bis} carried out by minimizing movements in the presence of elasticity we refer to  \cite{FFLM2-1,FFLM3,PP} where the setting of \emph{evolving graphs} is considered under an extra perturbative \emph{regularization term} added  to the equations. More precisely, in  \cite{FFLM2-1,FFLM3,PP} admissible film  profiles  are constrained to be parametrizable as  graph $\Gamma_h$ of  functions $h:Q\times [0,T]\rightarrow [0,+\infty)$ measuring the thickness of the film, where $Q:=(0,\ell)^{d-1}$ and $ [0,T]$ relate, for given  $\ell>0$ and a time $T>0$,  to the spatial and the temporal variable, respectively. 
Furthermore, in \cite{FFLM2-1,FFLM3,PP} a regularization term is added on  the basis of the argumentation already present in \cite{AG,H,H2} to overcome the fact that Equations \eqref{1bis} and \eqref{2bis} are backward parabolic in the presence of highly  anisotropic non-convex surface tension $\psi$, making the related Cauchy problem ill-posed. By adding such regularization term Equations \eqref{1bis} and \eqref{2bis} become
\begin{equation}\label{1tris}
  V=-\mathrm{div}_{\Gamma}(D\psi(\nu))-W(Eu)+\varepsilon R(k_1,k_2) \qquad \text{on $\Gamma_h$}
\end{equation}
and 
 \begin{equation}\label{2tris}
V = \Delta_{\Gamma}[\mathrm{div}_{\Gamma}(D\psi(\nu))+W(Eu)-\varepsilon R(k_1,k_2)] \qquad \text{on $\Gamma_h$}, 
\end{equation}  
respectively, where $\varepsilon>0$ and $R$ is a function depending on the principal curvatures $k_1,k_2$ of $\Gamma_h$. By defining $R$ (also on the basis of \cite{DGP})  as
\begin{equation}\label{reg_term}
    R(k_1,k_2):=\left( \Delta_{\Gamma}(|H|^{p-2}H)-|H|^{p-2}H\left(k_1^2+k_2^2-\frac{1}{p}H^2\right)\right)
\end{equation}
(where we recall that $H:=(k_1 + k_2)/2$)    with $p=2$ for $d=2$ in \cite{FFLM2-1,PP} and with $p>2$ for $d=3$ in \cite{FFLM3}, the backward parabolicity of the equation is avoided as surface patterns with large curvature get penalized. As such, minimizing movements are then used  in \cite{FFLM2-1,FFLM3,PP} so that  the \emph{gradient-flow structure} exhibited by Equations \eqref{1tris}  and \eqref{2tris} can be exploited:  by considering the functional 
$$
    \F(h,u_h):=\int_{\Omega_h}W(Eu)\de z+\int_{\Gamma_h}\left(\psi(\nu)+\frac{\varepsilon}{p}|H|^p\right)\de \mathcal{H}^2,
$$
where $u_h$ denotes the elastic equilibrium in $\Omega_h$ (under the proper periodic and boundary conditions)  and $\mathcal{H}^2$ denotes the two-dimensional Hausdorff measure, then Equations \eqref{1tris}  and \eqref{2tris} formally coincide with the gradient flow of $\F$ with respect to  an $L^2$- and $H^{-1}$-Riemannian structure. This allowed to establish short-time existence of a regular solution of \eqref{1tris} in  \cite{PP}  for $d=2$, and of \eqref{2tris}  in \cite{FFLM2-1} and \cite{FFLM3} for $d=2$ and  $d=3$, respectively. It remains open the case of \eqref{1tris} for $d=3$ that we intend here to tackle by also choosing, as in  \cite{FFLM3}, $p>2$ in  \eqref{reg_term}.

Therefore, by recalling that the normal velocity parametrized as the graph of the thickness functions $h$ is given by
\begin{equation}\label{Vparametrized}
 V=   \frac{1}{\sqrt{1+|Dh|^2}}\frac{\partial h}{\partial t}
\end{equation}
on $\Gamma_h$, where $Dh$ denotes the gradient with respect to the spatial coordinates,  
 the  \emph{Cauchy problem} under investigation depending on the period of time $T>0$ is the following:
\begin{align}\label{syst}
   \begin{cases}
      \frac{1}{\sqrt{1+|Dh|^2}}\frac{\partial h}{\partial t}=-\mathrm{div}_{\Gamma}(D\psi(\nu))-W(Eu)\\
   \hspace{2.5cm}+\varepsilon\left( \Delta_{\Gamma}(|H|^{p-2}H)-|H|^{p-2}H\left(|B|^2-\frac{1}{p}H^2\right)\right) &\quad\text{in $\R^2\times[0,T]$, }\\
    \mathrm{div}(\mathbb{C} Eu)=0 &\quad\text{in $\Omega_h$},\\
    \mathbb{C}Eu[\nu]=0 &\quad\text{on $\Gamma_h$},\\
    u(x,0,t)=(e^1_0x_1,e^2_0x_2,0),\\
    h(\cdot,t) \text{ and } Du(\cdot,t) \text{ are $Q$-periodic},\\
    h(\cdot,0)=h_0,
       \end{cases}
\end{align}
where 
$e_0:=(e_0^1,e_0^2)$ with $e_0^i>0$ is a vector which represents the mismatch between the crystalline lattices of the film and the substrate, $h_0$ is an admissible profile of the film at the initial time $t=0$, and spatial $Q$-periodic conditions are considered. We notice that the  Cauchy problem on the period of time $T>0$ considered in \cite{FFLM3} coincides with \eqref{syst} if we replace the first equation in \eqref{syst} with \eqref{2tris} (by taking into account \eqref{Vparametrized}). 

We can now more precisely detail the achieved result \PPP by \OOO referring \OOO  instead  to Section \ref{main_results} for \PPP the \OOO full characterization:  In Theorem \ref{teoconvfin} we establish the existence of a time $T_0>0$ for which the Cauchy problem \eqref{syst} admits a  solution in $[0,T_0]$, which is said to be \emph{regular} in the meaning that the first equation in \eqref{syst} is satisfied a.e.\ in $\R^2\times[0,T_0]$ (see Definition \ref{def_sol})\PPP.\OOO 

\subsection{Organization of the paper and description of the method} In Section \ref{sec2} we \PPP set \OOO the notation adopted throughout the paper. In Section \ref{sec3} we introduce the precise mathematical setting of the model, we introduce the minimizing-movement scheme obtained by discretizing the time interval and by defining the incremental family of minimum problems \eqref{min-prob} at each discrete time\PPP, and we state the main theorem of the paper (see Theorem \ref{teoconvfin}). \OOO In Section \ref{sec4}  we  \PPP begin by showing in Theorem \ref{teo-Min} \OOO the existence of minimizers for \eqref{min-prob}  so that a discrete-time evolution \PPP of the type of \OOO Definition \ref{disc-sol} can be constructed, \PPP and \OOO then we analyze \PPP its \OOO convergence properties in Theorems \ref{teo_h_conv} and \ref{teo_h_conv2}, \PPP ehich \OOO then are improved in Theorem  \ref{thm-t0} by selecting a specific time $T_0>0$.  In Section \ref{sec5} we prove in Theorem \ref{teoconvfin} the existence of a regular solution in the time interval $[0,T_0]$ by passing, in view of the previously proven convergence properties, to  the discrete \emph{Euler-Lagrange equation} \eqref{EL2}  satisfied by the minimizers of the incremental minimum problems \eqref{min-prob} (see Lemma \ref{EL-esp}), from which in Theorem \ref{boundD2H} we are able  to deduce a crucial uniform bound on the mean curvature \PPP needed to reach Theorem \ref{teoconvfin}. \OOO 
Finally, in Section \ref{sec7} we collect some auxiliary results often used in the paper for the Reader's convenience.

\bigskip\bigskip

\section{Notation}\label{sec2}
In this section we set the main notation used throughout the paper. The space of $m\times d$ matrices with real entries is denoted by $\R^{m\times d}$, and, in case $m=d$, the subspace of symmetric matrices is denoted by $\mathbb R^{d\times d}_{sym}$. Given a function $u\colon\R^d\to\R^m$, we denote its Jacobian matrix by $Du$, whose components are $(Du)_{ij}:= \partial_j u_i$ for $i=1,\dots,m$ and $j=1,\dots,d$, and when $u\colon \R^d\to\R^d$, we also denote by $Eu$ the symmetric part of the gradient, i.e., $Eu:=\frac{1}{2}(D u+D u^T)$. Given a tensor field $A\colon \R^d\to\R^{m\times d}$, by $\div A$ we mean its divergence with respect to the rows, namely $(\div A)_i:= \sum_{j=1}^d\partial_jA_{ij}$ for $i=1,\dots,m$. 

The norm of a generic Banach space $X$ is denoted by $\|\cdot\|_X$ and in case $X$ is a Hilbert space, we denote by $\langle\cdot,\cdot\rangle_X$ its inner product. In the case $X=\R^d$, we simplify the notation by using the symbol $\langle\cdot,\cdot \rangle$ to denote the euclidean scalar product, that is $\langle v,w\rangle:=\sum_{i=1}^d v_iw_i$ for every $v,w\in\R^d$, while $A:B:=\sum_{i,j=1}^d a_{ji}b_{ij}$ denotes the Hilbert-Schmidt product of two matrices $A,B\in \R^{d\times d}$. Given two matrices $A\in \R^{m\times d}$ and $B\in \R^{n\times q}$ we denote by $A\otimes B\in\R^{mn\times dq}$ the Kronecker product between the two matrices $A$ and $B$, defined as the block matrix $A\otimes B:=(a_{ij}B)_{i,j}$. Given two Banach spaces $X_1$ and $X_2$, the space of linear and continuous maps from $X_1$ to $X_2$ is denoted by $\mathscr L(X_1;X_2)$. For any $\mathbb A\in\mathscr L(X_1;X_2)$ and $u\in X_1$, we indicate the image of $u$ under $\mathbb A$ with $\mathbb A u\in X_2$. 

We denote the $d$-dimensional Lebesgue measure by $\mathcal L^d$ and the $(d-1)$-dimensional Hausdorff measure by $\mathcal H^{d-1}$. Given a bounded open set $\Omega$ with Lipschitz boundary, $\nu$ indicates the outer unit normal vector of $\partial\Omega$, which is defined $\mathcal H^{d-1}$-a.e.\, and we employ the usual definition of Lebesgue and Sobolev spaces on $\Omega$. The values of Sobolev functions on the boundary of their set of definition are always intended in the sense of traces. Given a set $U$, then $U^2:=U\times U$ with $\times$ denoting the cartesian product. Finally, given an open interval $(a,b)\subset\R$ and $p\in[1,\infty]$, we denote by $L^p(a,b;X)$ the space of $L^p$ functions from $(a,b)$ to $X$. We use $H^k(a,b;X)$ and $W^{k,p}(a,b;X)$ to denote the Sobolev space of functions from $(a,b)$ to $X$ with $k$ weak derivatives in $L^2(a,b;X)$ and $L^p(a,b;X)$, respectively.

\bigskip\bigskip

\section{Mathematical setting and main results}\label{sec3}
In this section we introduce the model, the main definitions, and the \PPP formal statement of the main result. \OOO 

\subsection{The variational model} Let us define $Q:=(0,\ell)^2$ with $\ell>0$. For $p>2$ we denote by $\Ws$ the space of all functions of $\W$ whose $Q$-periodic extension belong to $W^{2,p}_{loc}(\R^2)$, and we define the class of admissible profiles as
\begin{equation*}
    AP:=\{h\in\Ws:h(x)\geq 0\},
\end{equation*}
where we use the notation $x=(x_1,x_2)\in Q$.

Moreover, for $h\in\Ws$ we refer to the sets
\begin{equation}\label{defoh}
    \Gh:=\{z=(x,h(x)):x\in Q\}\text{ and }\Oh:=\{z=(x,y)\in Q\times \R:0<y<h(x)\}
\end{equation}
as the \emph{film profile} and the \emph{region of the film with height} $h$, respectively, while the corresponding sets with $Q$ replaced by $\R^2$ are denoted by $\Ghh$ and $\Ohh$. We define the \emph{family of periodic displacements} by
\begin{equation*}
    PD:=\{u:\Ohh\rightarrow\R^3:u(x,y)=u(x+\ell k,y)\text{ for every }(x,y)\in\Ohh \text{ and }k\in \Z^2\}
\end{equation*}
and the \emph{the family of admissible displacements} by
\begin{equation*}
    \Ad:=\{u\in L^2_{loc}(\Ohh,\R^3)\cap PD:u(x,0)=(e_0^1x_1,e_0^2x_2,0),\text{ } Eu_{|\Oh}\in L^2(\Oh,\R^3)\},
\end{equation*}
where $e_0:=(e_0^1,e_0^2)$, with $e_0^1,e_0^2>0$, is the vector representing the mismatch between the free-standing equilibrium crystalline lattice of the film and the substrate materials. Consequently, the \emph{family of admissible configurations} is
\begin{equation*}
    X:=\{(h,u):h\in AP\text{ and } u\in AD_h\}.
\end{equation*}
The \emph{elastic energy density} $W:\R^{2\times 2}_{sym}\rightarrow [0,\infty)$ is defined by
\begin{equation*}
    W(M):=\frac{1}{2}\C M:M,
\end{equation*}
where $\C$ is a fourth-order tensor such that there exists $k>0$ for which
\begin{equation*}
    \C M:M\geq 2k M:M \quad \text{for every $M\in \R^{2\times 2}_{sym}$}.
\end{equation*}

Furthermore, let $\psi:\R^3\rightarrow[0,+\infty)$ be a function of class $C^2$ on $\R^3\setminus\{0\}$ and positively one-homogeneous, which satisfies
\begin{equation}\label{fins}
    \frac{1}{c}|\xi|\leq\psi(\xi)\leq c|\xi|
\end{equation}
for every $\xi\in\R^3$ and for some positive constant $c$.

The \emph{configuration energy functional} $\F:X\rightarrow[0,\infty]$ is given by
\begin{equation*}
    \F(h,u):=\mathcal{W}(h,u)+\mathcal{S}(h)
\end{equation*}
for every admissible configuration $(h,u)\in X$, where $\mathcal{W}:X\rightarrow[0,\infty]$ and $\mathcal{S}:AP\rightarrow[0,\infty]$ represent the \emph{elastic} and the \emph{surface energy}, respectively. The elastic energy is defined by
\begin{equation*}
    \mathcal{W}(h,u):=\int_{\Oh}W(Eu)\de z 
\end{equation*}
while the surface energy is defined by
\begin{equation*}
    \mathcal{S}(h):=\int_{\Gh}\left(\psi(\nu)+\frac{\varepsilon}{p}|H|^p\right)\de \H^2,
\end{equation*}
where $\nu$ is the outer unit normal to $\Oh$, $\varepsilon$ is a positive constant, and $H=\div_{\Gh}\nu$ denotes the sum of the principal curvatures of $\Gh$, i.e.,
\begin{equation}\label{def-cur}
    H=-\div\left(\frac{Dh}{\sqrt{1+|Dh|^2}}\right)\quad \text{ in $Q.$}
\end{equation}
Notice that from \eqref{def-cur} it follows that
\begin{equation*}
    \int_Q H\de x=0.
\end{equation*}

\begin{definition}[Elastic equilibrium] We say that $\Bar{u}\in AD_h$ is the \emph{elastic equilibrium} of $\Bar{h}\in AP$ if
\begin{equation*}
    \F(\Bar{h},\Bar{u})=\min\{\F(\Bar{h},u):(\Bar{h},u)\in X\}.
\end{equation*}    
\end{definition}
\noindent Notice that the elastic equilibrium exixts for every $\Bar{h}\in AP$ and it is unique in view of the Dirichlet condition.


\subsection{The incremental minimum problem} Let $(h_0, u_0)\in X$ be such that $h_0>0$ and $u_0$ is the elastic equilibrium of $h_0$, and let $\Lambda_0$ be a positive constant such that
\begin{equation}\label{boundDh}
    \|h_0\|_{C^1_{\#}(Q)}<\Lambda_0.
\end{equation}
By considering a sequence $\tau_N\searrow 0$, for every $i\in \N$ we define inductively $(\hi,\ui)$ as a solution to the following minimum problem:
\begin{equation}\label{min-prob}
    \min\{\Gi(h,u):(h,u)\in X, \text{ }\|Dh\|_{\LiQ}\leq \Lambda_0\}.
\end{equation}
The functional $\Gi$ is given by
\begin{equation}\label{fun-G}
    \Gi(h,u):=\F(h,u)+\Pi(h),
\end{equation}
with the \emph{penalization term} $\Pi$ defined by
\begin{equation}\label{pen}
    \Pi(h):=\frac{1}{2\tau_N}\int_{\Gamma_{\him}}\left(\frac{h-\him}{\Jim}\right)^2\de\H^2=\frac{1}{2\tau_N}\int_{Q}\frac{(h-\him)^2}{\Jim}\de x,
\end{equation}
where $\Jim:=\sqrt{1+|D\him|^2}$. It is proved in Theorem \ref{teo-Min} that problem \eqref{min-prob} admits a minimizer. 

\begin{definition}[Dicrete-time evolution]\label{disc-sol}
Let $(h_0,u_0)\in X$ be an initial configuration and for $i,N\in\N$ let $(\hi,\ui)$ be a solution to \eqref{min-prob}. We refer to the piecewise linear interpolation $h_N:\R^2\times \R\rightarrow[0,\infty)$, given by
\begin{equation}\label{linear_h}
    h_N(x,t):=\hi(x)+\frac{1}{\tau_N}(t-(i-1)\tau_N)(\hi(x)-\him(x))
\end{equation}
for every $(x, t)\in  \R^2 \times [(i-1)\tau_N,i\tau_N]$ and $i\in\N$, as a \emph{discrete-time evolution} of the incremental minimum problem \eqref{min-prob}. Moreover, we denote with $u_N(\cdot,t)$ the elastic equilibrium corresponding to $h_N(\cdot,t)$.
\end{definition}

In Theorem \ref{teoconvfin} we make use of a different type of interpolation of solution to \eqref{min-prob}. 

\begin{definition}
    Let $(h_0,u_0)\in X$ be an initial configuration. Given a solution $(\hi,\ui)$ for $i,N\in\N$ to \eqref{min-prob}, we define the piecewise constant interpolations $\tilde{h}_N:\R^2\times \R\rightarrow[0,\infty)$ and $\tilde{u}_N:\R^3\times \R\rightarrow \R^3$ in the following way:
\begin{align}
    &\tilde{h}_N(x,t):=\hi(x)&&\text{for }(x, t)\in  \R^2 \times [(i-1)\tau_N,i\tau_N)\text{ and }i\in\N,\label{constant_h}\\
    &\tilde{u}_N(x,y,t):=\ui(x,y) &&\text{for }(x,y,t)\in\Omega_{\hi}\times[(i-1)\tau_N,i\tau_N)\text{ and }i\in\N\label{constant_u}
\end{align}
\end{definition}

\PPP \subsection{Main result}\label{main_results} 
We begin by providing a  formalization the notion of \emph{regular solutions} for the evolution  problem \eqref{syst}.  \OOO

\begin{definition}[Solutions of the evolution problem]\label{def_sol}
Let $(h_0,u_0)\in X$ be a configuration satisfying \eqref{boundDh} and $T>0$. We say that a function $h\in L^{\infty}(0, T ; \Ws)\cap H^1(0, T ; L^2_{\#}(Q))$ such that 
\begin{enumerate}[label=\textnormal{(\roman*)}]
\item $h(\cdot, 0) = h_0(\cdot)$ in $Q$,
    \item $-\div_{\Gamma}(D\psi(\nu))-W(Eu)+\varepsilon\left(\Delta_{\Gamma}(|H|^{p-2}H)-|H|^{p-2}H\left(|B|^2-\frac{1}{p}H^2\right)\right)\in L^2(0,T;L^2_{\#}(Q))$,
    \item the equation 
    \begin{equation*}
    \frac{1}{J}\frac{\partial h}{\partial t}=-\div_{\Gamma}(D\psi(\nu))-W(Eu)+\varepsilon\left(\Delta_{\Gamma}(|H|^{p-2}H)-|H|^{p-2}H\left(|B|^2-\frac{1}{p}H^2\right)\right)
\end{equation*}
is satisfied for $\mathcal{L}^3$-a.e.\ in $Q\times (0,T)$,
\end{enumerate}
where $J:=\sqrt{1+|Dh|^2}$, $\Gamma:=\Gamma_{h(\cdot,t)}$, $u(\cdot, t)$ is the elastic equilibrium in $\Omega_{h(\cdot,t)}$, and $W(Eu)$ denotes the trace of $W(Eu)$ on $\Gamma_{h(\cdot,t)}$, is a \emph{regular solution} to \eqref{syst} in $[0,T]$ with initial datum $h_0$.
\end{definition}

\PPP The main result of the manuscript asserts that a regular solution of \eqref{syst} exists for a short period of time. \OOO

\begin{theorem}[Short time existence of a regular solution]\label{teoconvfin}
    Let $h_0\in AP$ with $h_0>0$ and such that $h_0$ satisfies \eqref{boundDh}. There exist $T_0>0$ and a regular solution $h$ to \eqref{syst} in $[0,T_0]$ with initial datum $h_0$ in the sense of Definition \eqref{def_sol}. Moreover, there exist a nonincreasing function $g$ and a negligible set $Z_0$ such that
    \begin{equation}\label{ugF}
        \F(h(\cdot,t),u_h(\cdot,t))=g(t) \quad \text{for every $t\in[0,T_0]\setminus Z_0$}
    \end{equation}
    and
    \begin{equation}\label{disF}
        \F(h(\cdot,t),u_h(\cdot,t))\leq g(t) \quad \text{for every $t\in Z_0$}.
    \end{equation}
\end{theorem}

\noindent We notice that the solution $h$ and the time $T_0$ are given by Theorem \ref{teo_h_conv} and Theorem \ref{thm-t0}, respectively. 

\PPP Finally, let us conclude the section by observing that the solution $h$ to \eqref{syst} provided by Theorem \ref{teoconvfin} is said to be a \textit{variational solution} in order to indicate that  there exist a time step $\tau_N\searrow 0$ and a subsequence $\{h_{N_k}\}$ of a discrete-time evolution $\{h_N\}$ of the minimum incremental problem \eqref{min-prob} such that
  $$h_{N_k}\xrightharpoonup[k\to\infty]{}h\text{ in $H^1(0,T_0\LdQ)$}$$
    and 
    $$h_{N_k}\xrightarrow[k\to\infty]{}h\text{ in $C^{0,\beta}([0,T_0],C^{1,\alpha}_{\#}(Q))$}$$
  for every $\alpha\in(0,\frac{p-2}{p})$ and $\beta\in[0,\frac{(p+2)(p-2-\alpha p)}{8p^2})$. \OOO

\bigskip\bigskip

\section{Incremental minimum problem: Existence and convergence}\label{sec4}

In this section, we construct by applying a minimizing movement scheme a candidate to be a solution to \eqref{syst}. In Theorem \ref{teo-Min} we solve the incremental minimum problem. Then, in Theorem \ref{teo_h_conv} and \ref{teo_h_conv2} we prove that, up to a subsequence, the discrete-time evolution $\{h_N\}$ converges to some function $h$ as $N\to\infty$. Finally, in Theorem \ref{thm-t0} we choose $T_0$ small enough to get the validity of \eqref{stim-grad} and that $h_N$ is nonnegative.

The following result allows to verify that the incremental minimum problem is well defined, as, for each $i\in \N$, we can recursively find a solution to the minimum problem \eqref{min-prob}.

\begin{theorem}\label{teo-Min}
Let $(h_0,u_0)\in X$ be an initial configuration such that $h_0$ satisfies \eqref{boundDh}.	Then, the minimum problem \eqref{min-prob} admits a solution $(\hi,\ui)\in X$ for every $i\in\N$.
\end{theorem}
\begin{proof}
We prove the theorem by induction. We fix $i\in\N$ such that $i>1$ and we suppose to have a solution to \eqref{min-prob} for $k=1,...,i-1$. We construct a solution to \eqref{min-prob} for $k=i$. Firstly, we begin by noticing that, in view of the minimality of $(\hk,\uk)$, by \eqref{fun-G} and \eqref{pen} we have that
\begin{equation}\label{dis-G}
    \F(\hk,\uk)\leq \Gk(\hk,\uk)\leq\Gk(\hkm,\ukm)=\F(\hkm,\ukm),
\end{equation}
from which we deduce that
\begin{align}\label{bound-en}
    0\leq \inf\{\Gi(h,u):(h,u\in X)\}&\leq \Gi(\him,\uim)\notag\\
    &=\F(\him,\uim)\leq...\leq \F(h_0,u_0).
\end{align}
As a consequence of \eqref{bound-en}, we can select a minimizing sequence $\{(h_n,u_n)\}\in X$ such that
\begin{equation*}
    \|Dh_n\|_{\LiQ}\leq \Lambda_0\quad\text{ and }\quad \sup\{\Gi(h_n,u_n):n\in\N\}<\infty.
\end{equation*}
By denoting with $H_n$ the sum of principal curvatures of $\Gamma_{h_n}$, we obtain that
\begin{equation*}
    \sup\{\|H_n\|^p_{\LpQ}:n\in\N\}\leq \frac{p}{\varepsilon}\sup\{\Gi(h_n,u_n):n\in\N\}<\infty,
\end{equation*}
and so $\{H_n\}$ is bounded in $\LpQ$. Therefore, by Lemma \ref{A.3} the sequence $\{h_n\}$ is bounded in $\Ws$. Then, up to a subsequence, $h_n\rightharpoonup h$ in $\Ws$ as $n\to\infty$, from which we deduce that $h_n\rightarrow h$ in $C^{1,\alpha}_{\#}(Q)$ for some $\alpha\in (0,1)$ as $n\to \infty$. It follows that
\begin{equation*}
    H_n\xrightharpoonup[n\to \infty]{\LpQ} H=-\div\left(\frac{Dh}{\sqrt{1+|Dh|^2}}\right).
\end{equation*}
Furthermore, by lower semicontinuity
\begin{equation}\label{conv-surf}
    \int_{\Gamma_h}\left(\psi(\nu)+\frac{\varepsilon}{p}|H|^p\right)\de \H^2\leq\liminf_{n\to\infty}\int_{\Gamma_{h_n}}\left(\psi(\nu_n)+\frac{\varepsilon}{p}|H_n|^p\right)\de \H^2,
\end{equation}
and by Fatou's Lemma we conclude that 
\begin{equation}\label{conv-pen}
    \Pi(h)\leq \liminf_{n\to\infty}\Pi(h_n).
\end{equation}
Now we extend the function $u_n$ to $A:=Q\times (-\infty,0)$ in the following way
\begin{equation*}
    u_n(x,y):=(e_0^1x_1,e_0^2x_2,0)\varphi(y)
\end{equation*}
for every $(x,y)\in A$, where $\varphi$ is a given cut-off function such that $\varphi(y):=1$ in $(-1,+\infty)$ and $\varphi(y):=0$ in $(-\infty,-2)$. We consider also the set
\begin{equation*}
    A_h:=\{(x,y)\in\R^3:y<h(x)\}.
\end{equation*}
Since 
\begin{equation*}
    \sup\left\{\int_{\Omega_{h_n}}|Eu_n|^2\de z:n\in\N\right\}<\infty,
\end{equation*}
by reasoning as in \cite[Proposition 2.2]{FFLM2}, we use Korn's inequality to get the existence of a subsequence of $\{u_n\}_n$, and of a function $u\in H^1_{loc}(A_h;\R^3)$ with $Eu\in L^2(A_h;\R^{3\times 3}_{sym})$, such that $u_n\rightharpoonup u$ in $H^1(D;\R^3)$ as $n\to \infty$, for every $D$ compactly contained in $A_h$. Then, $(h,u)\in X$ and we also have
\begin{equation}\label{conv-el}
    \int_{\Oh}W(Eu)\de z\leq \liminf_{n\to \infty}\int_{\Omega_{h_n}}W(Eu_n)\de z.
\end{equation}
Thanks to \eqref{conv-surf}, \eqref{conv-pen}, and \eqref{conv-el} we can conclude that $(h,u)$ is a minimizer of \eqref{min-prob} for $k=i$.
\end{proof}

We now prove that the discrete-time evolution $\{h_N\}$ is uniformly bounded in $L^{\infty}(0, T ; \Ws) \cap H^1(0, T ; L^2(Q))$. 

\begin{theorem}\label{teo_h_conv}
Let $(h_0,u_0)\in X$ be an initial configuration such that $h_0$ satisfies \eqref{boundDh}. For every $N,i\in \N$ we have
\begin{align}
    &\int_0^{+\infty}\int_Q\left|\frac{\partial h_N}{\partial t}\right|^2\de x\de t\leq CF(h_0,u_0),\label{HuLd}\\
    & \F(\hi,\ui)\leq \F(\him,\uim)\leq \F(h_0,u_0),\label{en-dec}\\
    &\sup_{t\in[0,+\infty)}\|h_N(\cdot,t)\|_{\Ws}<+\infty\label{wdp},
\end{align}
for some constant $C=C(\Lambda_0)>0$. Moreover, up to a subsequence, for every $T>0$ we have the following convergences
\begin{align}
 &h_N\xrightharpoonup[N\to\infty]{}h\text{ in $H^1(0,T,\LdQ)$},\label{convhu}\\
    &h_N\xrightarrow[N\to\infty]{}h\text{ in $C^{0,\alpha}([0,T],\LdQ)$ for every $\alpha \in (0,\tfrac{1}{2})$},\label{convcz}
\end{align}
where the function $h$ satisfies
\begin{equation}\label{ineq-ene}
     h(\cdot,t)\in\Ws\text{ and } \F(h(\cdot,t),u_{h(\cdot,t)})\leq \F(h_0,u_0) \text{ for every $t\in[0,+\infty)$}.
\end{equation}
\end{theorem}
\begin{proof}
We begin by observing that \eqref{en-dec} is a direct consequence of \eqref{dis-G}.

To prove \eqref{HuLd} we notice that for each $i\in \N$, by \eqref{dis-G} we have that
\begin{equation*}
   \Gi(\hi,\ui)\leq\Gi(\him,\uim)=\F(\him,\uim),
\end{equation*}
which implies 
\begin{equation}\label{pen-ineq}
    \Pi(\hi)\leq \F(\him,\uim)-\F(\hi,\ui).
\end{equation}
Since $\|D\him\|_{\LiQ}<\Lambda_0$, by summing in \eqref{pen-ineq} over $i$ we get
\begin{align}\label{sum}
  \sum_{i=0}^{+\infty}\tau_N\int_Q\left(\frac{\hi-\him}{\tau_N}\right)^2\de x&\leq C(\Lambda_0)\sum_{i=0}^{+\infty}[\F(\him,\uim)-\F(\hi,\ui)]\notag\\
  &=C(\Lambda_0)\F(h_0,u_0),
\end{align}
with $C(\Lambda_0):=2\sqrt{1+\Lambda_0^2}$. Furthermore, by \eqref{linear_h} we have 
\begin{equation*}
    \frac{\partial h_N}{\partial t}=\frac{\hi-\him}{\tau_N},
\end{equation*}
and hence, by also \eqref{sum} we obtain \eqref{HuLd}. 

To prove \eqref{wdp}, we notice that \eqref{en-dec} implies
\begin{equation}\label{Hb}
    \sup\left\{\int_{\Gamma_{\hi}}|H_{i,N}|^p\de\H^2:i,N\in\N\right\}<\frac{p}{\varepsilon}\F(h_0,u_0),
\end{equation}
from which, by taking into account that $\|D\hi\|_{\LiQ}<\Lambda_0$ and by Lemma \ref{A.3}, \eqref{wdp} follows.

Notice that \eqref{convhu} is a direct consequence of \eqref{HuLd} and \eqref{wdp}. It remains to prove \eqref{convcz} and \eqref{ineq-ene}. Since $h_N(x,\cdot)$ is absolutely continuous on $[0,T]$ for every $T>0$, then by Holder inequality, Fubini's Theorem, and \eqref{HuLd}, for every $t_1,t_2\in[0,T]$ we obtain
\begin{align}
    \|h_N(\cdot,t_1)-h_N(\cdot,t_2)\|_{\LdQ}&\leq \left(\int_Q\left(\int_{t_1}^{t_2}\frac{\partial h_N}{\partial t}(x,t)\de t\right)^2\de x\right)^{\frac{1}{2}}\nonumber\\
    &\leq \left(\int_{t_1}^{t_2}\left\|\frac{\partial h_N}{\partial t}(\cdot,t)\right\|_{\LdQ}^2\de t\right)^{\frac{1}{2}}(t_2-t_1)^{\frac{1}{2}}\notag\\
  &\leq \sqrt{CF(h_0,u_0)}(t_2-t_1)^{\frac{1}{2}}.
    \label{hofu}
\end{align}
By using \eqref{hofu} and Ascoli-Arzelà Theorem (see e.g. \cite[Proposition 3.3.1]{A}) we get \eqref{convcz}. 

Finally, by \eqref{en-dec}, \eqref{convhu}, \eqref{convcz}, and the lower semicontinuity we obtain \eqref{ineq-ene}.
\end{proof}

In the following result the convergence of $h_N$ to $h$ is significantly improved, and we prove also that $\tilde{h}_N$ converges to $h$.

\begin{theorem}\label{teo_h_conv2}
Let $(h_0,u_0)\in X$ be an initial configuration such that $h_0$ satisfies \eqref{boundDh}. For every $T>0$ we have the following convergences
\begin{align}
    &h_N\xrightarrow[N\to\infty]{}h\text{ in $C^{0,\beta}([0,T],C^{1,\alpha}_{\#}(Q))$}\label{conv-holder},\\
    &\tilde{h}_N\xrightarrow[N\to\infty]{}h\text{ in $L^{\infty}(0,T,C^{1,\alpha}_{\#}(Q))$}\label{conv-inf},
\end{align}
for every $\alpha \in (0,\frac{p-2}{p})$ and $\beta\in[0,\frac{(p-2-\alpha p)(p+2)}{8p^2})$. Moreover, $h(\cdot,t)\rightarrow h_0$ in $C^{1,\alpha}_{\#}(Q)$ as $t\to 0^+$.
\end{theorem}

\begin{proof}
    Let $t_1,t_2\in[0,T]$ be such that $t_1<t_2$, then by \eqref{wdp}, \eqref{hofu}, and the interpolation inequality of Lemma \ref{A.6} we have
    \begin{align}\label{stima-grad}
        \|Dh_N(\cdot,t_2)&-Dh_N(\cdot,t_1)\|_{\LiQ}\nonumber\\
        &\leq C\|D^2h_N(\cdot,t_2)-D^2h_N(\cdot,t_1)\|_{\LpQ}^{\frac{p+2}{2p}}\|h_N(\cdot,t_2)-h_N(\cdot,t_1)\|_{\LpQ}^{\frac{p-2}{2p}}\nonumber\\
        &\leq C\|h_N(\cdot,t_2)-h_N(\cdot,t_1)\|_{\LpQ}^{\frac{p-2}{2p}}\nonumber\\
        &\leq C\left(\|D^2h_N(\cdot,t_2)-D^2h_N(\cdot,t_1)\|^{\frac{p-2}{2p}}_{\LdQ}\|h_N(\cdot,t_2)-h_N(\cdot,t_1)\|_{\LdQ}^{\frac{p+2}{2p}}\right)^{\frac{p-2}{2p}}\nonumber\\
        &\leq C(t_2-t_1)^{\frac{p^2-4}{8p^2}}.
    \end{align}

Moreover, thanks to Mean Value Theorem there exists $x_0\in Q$ such that
\begin{equation*}
    h_N(x_0,t_2)-h_N(x_0,t_1)=\frac{1}{\ell^2}\int_Q(h_N(x,t_2)-h_N(x,t_1))\de x,
\end{equation*}
from which we deduce that for every $x\in Q$ we have
\begin{align}\label{stima-punt}
    |h_N(x,t_2)-h_N(x,t_1)|\leq &\ell^2\|Dh_N(\cdot,t_2)-Dh_N(\cdot,t_1)\|_{\LiQ}
    \notag\\
  &\qquad+\frac{1}{\ell}\|h_N(\cdot,t_2)-h_N(\cdot,t_1)\|_{\LdQ}.
\end{align}
Now, by \eqref{hofu}, \eqref{stima-grad}, and \eqref{stima-punt} we get
\begin{equation}\label{stima-h}
  \|h_N(\cdot,t_2)-h_N(\cdot,t_1)\|_{\LiQ}\leq C\left[(t_2-t_1)^{\frac{p^2-4}{8p^2}}+(t_2-t_1)^{\frac{1}{2}}\right].
\end{equation}

Furthermore, notice that for every $\alpha \in (0,\frac{p-2}{p})$ we have
\begin{align}\label{stima-hold}
    [Dh_N(\cdot,t_2)-Dh_N(\cdot,t_1)]_{\alpha}\leq &[Dh_N(\cdot,t_2)-Dh_N(\cdot,t_1)]^{\frac{\alpha p}{p-2}}_{\frac{p-2}{p}}\,\times\notag\\
  &\qquad\left(2\|Dh_N(\cdot,t_2)-Dh_N(\cdot,t_1)\|_{\LiQ}\right)^{\frac{p-2-\alpha p}{p-2}}
\end{align}
where $[\cdot]_{\alpha}$ denotes the $\alpha$-Holder seminorm, and by Sobolev embedding and \eqref{wdp} we obtain
\begin{equation}\label{sob}
    \sup\left\{\sup_{t\in[0,T]}\|h_N(\cdot,t)\|_{C^{1,\frac{p-2}{p}}_{\#}(Q)}:\text{ }N\in\N\right\}<+\infty.
\end{equation}
By \eqref{stima-grad}, \eqref{stima-hold}, and \eqref{sob} we can write
\begin{equation}\label{stima-grad-hold}
    [Dh_N(\cdot,t_2)-Dh_N(\cdot,t_1)]_{\alpha}\leq C (t_2-t_1)^{\frac{(p+2)(p-2-\alpha p)}{8p^2}}.
\end{equation}

Therefore, it follows from \eqref{stima-grad}, \eqref{stima-h}, and \eqref{stima-grad-hold}, that for every $\alpha \in(0,\frac{p-2}{p})$, $h_N$ is uniformly equicontinuous with respect to the $C^{1,\alpha}(Q)$-norm topology and that
\begin{equation*}
    \|h_N(\cdot,t_2)-h_N(\cdot,t_1)\|_{C^{1,\alpha}(Q)}\leq C(t_2-t_1)^{\frac{(p+2)(p-2-\alpha p)}{8p^2}}.
\end{equation*}
In particular, by applying Ascoli-Arzelà Theorem we get both \eqref{conv-holder} and \eqref{conv-inf}. Finally, by noticing that $\|h_N(\cdot,t)-h_N(\cdot,t_1)\|_{C^{1,\alpha}(Q)}\to 0$ as $t\to t_1$, we choose $t_1=0$ to conclude the proof.
\end{proof}

From now on, we assume that the initial profile $h_0$ is strictly positive. Thanks to this, we can use standard elliptic estimates to establish the convergence of $u_N$ and $\tilde{u}_N$.  
\begin{theorem}\label{thm-t0}
    Let $(h_0,u_0)\in X$ be an initial configuration such that $h_0>0$ satisfies \eqref{boundDh}. Then there exist $T_0>0$ and $C>0$ depending only on $(h_0,u_0)$ such that:
    \begin{enumerate}[label=\textnormal{(\roman*)}]
     \item \label{i} $h_N,h\geq C_0>0$ for some positive constant $C_0$, and
    \begin{equation}\label{stim-grad}
        \sup_{t\in[0,T_0]}\|Dh_N(\cdot,t)\|_{\LiQ}<\Lambda_0
    \end{equation}
    for every $N\in\N$;
        \item it holds
        \begin{equation}\label{stima-grad-u}
           \|D\ui\|_{C^{0,\frac{p-2}{p}}(\Bar{\Omega}_{\hi};\R^{3\times 3}_{sym})}\leq C;
        \end{equation}
        \item for every $\alpha\in(0,\frac{p-2}{p})$ and $\beta\in[0,\frac{(p+2)(p-2-\alpha p)}{8p^2})$ we have the following convergences:
        \begin{align}
            &E(u_N(\cdot,h_N))\xrightarrow[N\to \infty]{}E(u(\cdot,h)) \text{ in $C^{0,\beta}([0,T_0];C^{1,\alpha}_{\#}(Q))$}\label{conv-trac-u},\\
            &E(\tilde{u}_N(\cdot,\tilde{h}_N))\xrightarrow[N\to \infty]{}E(u(\cdot,h)) \text{ in $L^{\infty}(0,T_0;C^{1,\alpha}_{\#}(Q))$}\label{conv-trac-tu},
        \end{align}
        where $u(\cdot,t)$ is the elastic equilibrium in $\Omega_{h(\cdot,t)}$.
    \end{enumerate}
\end{theorem}
\begin{proof}
    We begin by observing that by inserting $t_2=t$ and $t_1=0$ in \eqref{stima-h} we can write
    \begin{equation*}
        h_N(x,t)\geq \min_{x\in Q} h_0(x)-C\left[t^{\frac{p^2-4}{8p^2}}+t^{\frac{1}{2}}\right]\qquad \text{for every $(x,t)\in Q\times [0,T]$}.
    \end{equation*}
    Therefore, since $t\mapsto t^{\frac{p^2-4}{8p^2}}+t^{\frac{1}{2}}$ is an increasing function, there exists $T_1>0$ such that $h_N(x,t)\geq C_0$ for every $(x,t)\in Q\times [0,T_1]$ and for some constant $0<C_0<\min\{h_0(x):x\in Q\}$. As a consequence $h\geq C_0$. Now, we notice that in view of \eqref{stima-grad} with $t_2=t$ and $t_1=0$ we get
    \begin{equation*}
        \|Dh_N(\cdot,t)\|_{\LiQ}\leq \|Dh_N(\cdot,t)-Dh_0(\cdot)\|_{\LiQ}+\|Dh_0\|_{\LiQ}\leq Ct^{\frac{p^2-4}{8p^2}}+\|h_0\|_{C^1_{\#}(Q)},
    \end{equation*}
    and hence, by \eqref{boundDh} we can choose $T_2>0$ such that 
    \begin{equation*}
        T_2<\left(\frac{\Lambda_0-\|h_0\|_{C^1_{\#}(Q)}}{C}\right)^{\frac{8p^2}{p^2-4}},
    \end{equation*}
    from which we deduce that
    \begin{equation*}
        \sup_{t\in[0,T_2]}\|Dh_N(\cdot,t)\|_{\LiQ}<\Lambda_0.
    \end{equation*}
   By choosing $T_0:=\min\{T_1,T_2\}$ we conclude the proof of \ref{i}.

    To prove \eqref{stima-grad-u}, we use standard elliptic estimates to bound the norm of $D\ui$ in $C^{0,\alpha}(\Bar{\Omega}_{\hi})$ with a constant depending only on the $C^{1,\alpha}$-norm of $\hi$ (see \cite{FM2}). Then, \eqref{sob} implies \eqref{stima-grad-u}.

    Finally, by \eqref{conv-holder}, \eqref{conv-inf}, and Lemma \ref{A.1}, \eqref{conv-trac-u} and \eqref{conv-trac-tu} follows.   
\end{proof}

\bigskip\bigskip

\section{Existence of evolution}\label{sec5}

In this section we prove that the candidate found in Section \ref{sec4} is a solution to \eqref{syst} in the sense of the Definition \ref{def_sol} in the case of short time intervals (see Theorem \ref{teoconvfin}). 

In the next lemma, thanks to the regularity given by Theorem \ref{thm-t0}, we derive the Euler-Lagrange equation satisfied by the minimizer $(\hi,\ui)$ to the problem \eqref{min-prob}.

\begin{lemma}\label{EL-esp}
    Every minimizer $(\hi,\ui)$ of \eqref{min-prob} satisfies the following Euler-Lagrange equation
    \begin{align}\label{EL2}
    &\int_Q W(E(\ui(x,\hi(x))))\varphi\de x+\int_Q\langle D\psi(-D\hi,1),(-D\varphi,0)\rangle \de x\notag\\
  &+\frac{\varepsilon}{p}\int_Q |\Hi|^p\frac{\langle D\hi,D\varphi\rangle}{\Ji}\de x\nonumber\\
   & -\varepsilon\int_Q|\Hi|^{p-2}\Hi\Bigg[\Delta \varphi-\frac{\langle D^2\varphi D\hi,D\hi \rangle}{\Ji^2}-\frac{\Delta\hi \langle D\hi,D\varphi \rangle}{\Ji^2}\notag\\
  &\qquad\qquad\qquad\qquad\qquad\qquad\qquad\qquad\qquad\qquad\qquad-2\frac{\langle D^2\hi D\hi,D\varphi \rangle}{\Ji^2}\Bigg]\de x\nonumber\\
   &-3\varepsilon\int_Q|\Hi|^{p-2}\Hi\frac{\langle D\hi,D\varphi\rangle\langle D^2\hi D\hi,D\hi\rangle}{\Ji^4}\de x\notag\\
  &+\frac{1}{\tau_N}\int_Q\frac{\hi-\him}{\Jim}\varphi\de x=0,
\end{align}
for every $\varphi\in C_{\#}^2(Q)$.
\end{lemma}

\begin{proof}
Since $(\hi,\ui)$ is the minimizer of \eqref{min-prob}, it satisfies
    \begin{equation}\label{der-EL}
    \frac{\de}{\de s}\left(\F(\hi(s),\ui)+\frac{1}{2\tau_N}\int_{Q}\frac{(\hi(s)-\him)^2}{\Jim}\de x\right)_{|s=0}=0
\end{equation}
for every $\varphi\in C^2_{\#}(Q)$, where $\hi(s):=\hi+s\varphi$.

By considering the penalization, we have
\begin{equation}\label{der-pen}
   \frac{\de}{\de s}\left(\frac{1}{2\tau_N}\int_{Q}\frac{(\hi(s)-\him)^2}{\Jim}\de x\right)_{|s=0}=\frac{1}{\tau_N}\int_{Q}\frac{(\hi-\him)\varphi}{\Jim}\de x.
\end{equation}
For the elastic energy, we use the definition \eqref{defoh} and we get
\begin{align}\label{der-elas}
    &\frac{\de}{\de s}\left(\int_{\Omega_{\hi+s\varphi}}W(E\ui(x,y))\de x\de y\right)_{|s=0}= \frac{\de}{\de s}\left(\int_{Q}\int_0^{\hi(s)}W(E\ui(x,y))\de y\de x\right)_{|s=0}\nonumber\\
    &=\int_Q [W(E\ui(x,\hi(x,s))\varphi(x)]_{|s=0}\de x=\int_Q W(E\ui(x,\hi(x))\varphi(x)\de x.
\end{align}
It remains to consider the surface energy. We begin by observing that
\begin{align}\label{der-psi}
    &\frac{\de}{\de s}\left(\int_{Q}\psi(-D\hi(s),1)\de x\de y\right)_{|s=0}\nonumber\\
    &=\int_Q [\langle D\psi(-D\hi(s),1),(-D\varphi,0)\rangle]_{|s=0}\de x=\int_Q \langle D\psi(-D\hi,1),(-D\varphi,0)\rangle\de x,
\end{align}
and that the regularization term can be treated in the following way
\begin{align}\label{der-h}
    &\frac{\de}{\de s}\left(\int_{\Gamma_{\hi(s)}}|H_{\hi(s)}|^p\de\H^2 \right)_{|s=0}=\frac{\de}{\de s}\left(\int_{Q}\f(s)\g(s)\de x\de y \right)_{|s=0}\nonumber\\
   & =\int_Q\left[\frac{\de \f}{\de s}(0)\g(0)+\f(0)\frac{\de \g}{\de s}(0)\right]\de x\de y,
\end{align}
where 
\begin{align*}
   &\g(s):=\sqrt{1+|D\hi(s)|^2}\quad \text{ and }\quad \f(s):=\left|\div\left(\frac{D\hi(s)}{\g(s)}\right)\right|^p
\end{align*}
Notice that
\begin{equation}
 \frac{\de \g}{\de s}(s)=\frac{1}{\g(s)}\langle D\hi(s),D\varphi\rangle\Longrightarrow \frac{\de \g}{\de s}(0)=\frac{\langle D\hi ,D\varphi\rangle}{\Ji},\label{deriv-g}
\end{equation}
and that
\begin{align}
    \frac{\de \f}{\de s}(s)=p|H_{\hi(s)}|^{p-2}&H_{\hi(s)}\frac{\de }{\de s}H_{\hi(s)} \notag\\
  &\qquad\Longrightarrow \frac{\de \f}{\de s}(0)=p|\Hi|^{p-2}\Hi\frac{\de }{\de s}(H_{\hi(s)})_{|s=0}\label{der-f}.
\end{align}
By definition we have
\begin{align*}
    H_{\hi(s)}&=-\div\left(\frac{D\hi(s)}{\g(s)}\right)=-\sum_{k=1}^2\frac{\partial}{\partial x_k}\left(\frac{\frac{\partial \hi(s)}{\partial x_k}}{\g(s)}\right)\nonumber\\
    &=-\sum_{k=1}^2\left[\frac{1}{\g(s)}\frac{\partial^2 \hi(s)}{\partial x_k^2}+\frac{\partial \hi(s)}{\partial x_k}\frac{\partial}{\partial x_k}\left(\frac{1}{\g(s)}\right)\right]\nonumber\\
    &=-\sum_{k=1}^2\left[\frac{1}{\g(s)}\frac{\partial^2 \hi(s)}{\partial x_k^2}-\frac{1}{2\g^3(s)}\frac{\partial \hi(s)}{\partial x_k}\sum_{m=1}^2\frac{\partial}{\partial x_k}\left(\frac{\partial \hi(s)}{\partial x_m}\right)^2\right]\nonumber\\
    &=-\frac{\Delta \hi(s)}{\g(s)}+\frac{\langle D^2\hi(s)D\hi(s),D\hi(s)\rangle}{\g^3(s)}:=A(s)+B(s),
\end{align*}
and hence, we can write
\begin{align}\label{A+B}
    \frac{\de}{\de s}(H_{\hi(s)})_{|s=0}=\frac{\de A}{\de s}(0)+\frac{\de B}{\de s}(0).
\end{align}
Now, we compute the derivative of $A$ and $B$ with respect to $s$. Since
\begin{align*}
    \frac{\de A(s)}{\de s}=\frac{-\Delta\varphi\g(s)+\frac{\Delta\hi(s)}{\g(s)}\langle D\hi(s),D\varphi\rangle}{\g^2(s)}=\frac{-\Delta\varphi\g^2(s)+\Delta\hi(s)\langle D\hi(s),D\varphi\rangle}{\g^3(s)},
\end{align*}
we get
\begin{equation}\label{der-A}
    \frac{\de A}{\de s}(0)=-\frac{\Delta\varphi}{\Ji}+\frac{\Delta\hi\langle D\hi,D\varphi\rangle}{\Ji^3}.
\end{equation}
Moreover, by setting
\begin{equation*}
    C(s):=\sum_{k=1}^2\frac{\partial \hi(s)}{\partial x_k}\sum_{m=1}^2\frac{\partial^2 \hi(s)}{\partial x_k\partial x_m}\frac{\partial \hi(s)}{\partial x_m},
\end{equation*}
in view of \eqref{deriv-g} we have
\begin{align}\label{deB}
     \frac{\de B}{\de s}(0)&= \frac{\de }{\de s}\left(\frac{C(s)}{\g^3(s)}\right)_{|s=0}=\frac{\frac{\de C}{\de s}(0)\g^3(0)-C(0)\frac{\de }{\de s}(\g^3(s))_{|s=0}}{\g^6(0)}\nonumber\\
     &=\frac{1}{\Ji^3}\frac{\de C }{\de s}(0)-\frac{3\langle D^2\hi D\hi,D\hi\rangle\langle D\hi(s),D\varphi\rangle}{\Ji^5}.
\end{align}
It remains to study the last term $C$:
\begin{align*}
    \frac{\de C(s)}{\de s}&=\sum_{k=1}^2\frac{\partial \varphi}{\partial x_k}\sum_{m=1}^2\frac{\partial^2 \hi(s)}{\partial x_k\partial x_m}\frac{\partial \hi(s)}{\partial x_m}+\sum_{k=1}^2\frac{\partial \hi(s)}{\partial x_k}\sum_{m=1}^2\frac{\de}{\de s}\left(\frac{\partial^2 \hi(s)}{\partial x_k\partial x_m}\frac{\partial \hi(s)}{\partial x_m}\right)\nonumber\\
    &=\langle D^2\hi(s) D\hi(s),D\varphi \rangle +\sum_{k=1}^2\frac{\partial \hi(s)}{\partial x_k}\sum_{m=1}^2\left(\frac{\partial^2 \varphi}{\partial x_k\partial x_m}\frac{\partial \hi(s)}{\partial x_m}+\frac{\partial^2 \hi(s)}{\partial x_k\partial x_m}\frac{\partial \varphi}{\partial x_m}\right)\nonumber\\
    &=\langle D^2\hi(s) D\hi(s),D\varphi \rangle +\langle D^2\varphi D\hi(s),D\hi(s) \rangle +\langle D^2\hi(s) D\varphi,D\hi(s) \rangle \nonumber\\
    &=2\langle D^2\hi(s) D\hi(s),D\varphi \rangle+\langle D^2\varphi D\hi(s),D\hi(s) \rangle,
\end{align*}
and hence
\begin{equation}\label{deCzero}
    \frac{\de C}{\de s}(0)=2\langle D^2\hi D\hi,D\varphi \rangle+\langle D^2\varphi D\hi,D\hi \rangle.
\end{equation}
By \eqref{deB} and \eqref{deCzero} we obtain
\begin{align}\label{der-B}
  \frac{\de B}{\de s}(0)=&\frac{2\langle D^2\hi D\hi,D\varphi \rangle+\langle D^2\varphi D\hi,D\hi \rangle}{\Ji^3}\notag\\
  &\qquad\qquad-\frac{3\langle D^2\hi D\hi,D\hi\rangle\langle D\hi(s),D\varphi\rangle}{\Ji^5}.
\end{align}
Finally, from \eqref{der-EL}-\eqref{der-A} and \eqref{der-B} the assertion follows.
\end{proof}

Now let $\Htn:\R^2\times [0,T_0]\rightarrow \R$ be the function defined by 
\begin{equation}\label{Hn}
    \Htn(x,t):=\Hi(x,\hi(x))\quad \text{ for $t\in[(i-1)\tau_N,i\tau_N)$},
\end{equation}
where $\Hi$ is the sum of the principal curvatures of $\Gin$. Moreover, we will denote by $|\Bi|^2$ the sum of the squares of the principal curvatures of $\Gin$, and by $\Ji$ the following quantity
\begin{equation*}
    \Ji:=\sqrt{1+|D\hi|^2},
\end{equation*}

\begin{theorem}\label{boundD2H}
    Let $T_0$ be as in Theorem \ref{thm-t0} and let $\Htn$ be given by \eqref{Hn}. Then, there exists a constant $C>0$ such that
    \begin{equation}\label{stimaD2H}
        \int_0^{T_0}\int_Q|D^2(|\Htn|^{p-2}\Htn)|^2\de x\de t\leq C
    \end{equation}
for every $N\in\N$.
\end{theorem}

\begin{proof}
We divide the proof into two steps.
    
\textbf{Step 1} In this step we prove that for every $k\geq 1$ and $\sigma\in (0,\frac{1}{p-1})$ we have
\begin{equation*}
    |\Hi|^{p-2}\Hi\in W_{\#}^{1,q}(\Gin)\quad\text{ and }  \quad\hi\in C_{\#}^{2,\sigma}(Q).
\end{equation*}
In order to do this, we begin by showing that $\hi\in W_{\#}^{2,q}(Q)$ for every $q\geq 2$. Since $\hi$ is the solution to \eqref{min-prob}, thanks to Lemma \ref{EL-esp} it satisfies \eqref{EL2}. Now, by setting 
\begin{align}
    & w:=|\Hi|^{p-2}\Hi, \quad A:=\varepsilon\left(I-\frac{D\hi\otimes D\hi}{\Ji^2}\right),\label{A}\\
    & b:=\pi(D(\psi(-D\hi,1)))-\frac{\varepsilon}{p}|\Hi|^p D\hi\nonumber\\
    &\hspace{1cm}+\varepsilon w\left[-\frac{\Delta\hi  D\hi}{\Ji^2}-2\frac{ D^2\hi D\hi}{\Ji^2}+3\frac{\langle D^2\hi D\hi,D\hi\rangle D\hi}{\Ji^4}\right],\nonumber\\
    &c:=-W(E(\ui(x,\hi(x))))-\frac{\hi-\him}{\tau_N\Jim}\nonumber,
\end{align}
we can rephrase \eqref{EL2} as follows
\begin{equation}\label{eq-int}
    \int_Q w A:D^2\varphi\de x+\int_Q\langle b,D\varphi\rangle\de x+\int_Q c\varphi\de x=0 \qquad \text{for every $\varphi\in C^{\infty}_{\#}(Q)$}.
\end{equation}
By \eqref{wdp} and Theorem \ref{thm-t0} we have that $A\in W^{1,p}_{\#}(Q;\R^{2\times 2}_{sym})$, $b\in L^1(Q;\R^2)$ and $c\in C_{\#}^{0,\alpha}(Q)$ for some $\alpha$, and hence, since \eqref{eq-int} is in particular satisfied for every $\varphi\in C^{\infty}_{\#}(Q)$ with $\int_Q\varphi\de x=0$, we can apply Lemma \ref{A.2} to get that $w\in L^q(Q)$ for $q\in (\frac{p}{p-1},2)$. By following the same argument in \cite[Theorem 3.11]{FFLM3} we obtain $|\Hi|^{p-2}\Hi\in W^{1,q}_{\#}(Q)$ for every $q\geq 1$, then $|\Hi|^{p-1}\in C^{0,\alpha}_{\#}(Q)$ for every $\alpha\in (0,1).$ As a consequence, $\Hi\in C^{0,\sigma}_{\#}(Q)$ for every $\sigma\in (0,\frac{1}{p-1})$, and hence, by Schauder estimates we deduce that $\hi\in C^{2,\sigma}_{\#}(Q)$ for every $\sigma\in (0,\frac{1}{p-1})$.

\textbf{Step 2} In view of Lemma \ref{A.3} and the fact that $\|D\hi\|_{\LiQ}<\Lambda_0$, there exists a positive constant $C=C(q,\Lambda_0)$ such that
\begin{equation}\label{stimaD2h}
    \|D^2\hi\|_{L^q}(Q)\leq C\|\Hi\|_{L^q(Q)}.
\end{equation}
Furthermore, since $\Gin$ is of class $C^{2,\sigma}$ then $|\Hi|^{p-2}\Hi\in H^2(\Gin)$, and this implies that $|\Hi|^{p-2}\Hi\in H^2(Q)$. 

It is easy to verify (see \cite[Theorem 3.11]{FFLM3}) that $\hi$ satisfies the following Euler-Lagrange equation in intrinsic form
\begin{align}\label{intr}
   & -\varepsilon\int_{\Gin}D_{\Gin}(|\Hi|^{p-2}\Hi)D_{\Gin}\phi\de \H^2+\varepsilon\int_{\Gin}|\Hi|^{p-2}\Hi\left(|B_{i,N}|^2-\frac{1}{p}\Hi^2\right)\phi \de \H^2\nonumber\\
    &-\int_{\Gin}\left[\div_{\Gin}(D\psi(\nu_{i,N}))+W(E\ui)\right]\phi\de \H^2=\int_{\Gin}\frac{\hi-\him}{\Jim\tau_N}\phi\de \H^2,
\end{align}
where $\phi:=\frac{\varphi}{\Ji}\circ \pi$ with $\varphi\in C^2_{\#}(Q)$ such that $\int_Q\varphi\de x=0$. Now, by setting $w:=|\Hi|^{p-2}\Hi$ we can rephrase \eqref{intr} in the following way
\begin{align}\label{intr2}
    -\int_{Q}\langle ADw, D&\left(\frac{\varphi}{\Ji}\right)\rangle \Ji\de x+\varepsilon\int_{Q}w\varphi\left(|B_{i,N}|^2-\frac{1}{p}\Hi^2\right)\de x\nonumber\\
    &-\int_{Q}\left[\div_{\Gin}(D\psi(\nu_{i,N}))+W(E\ui)\right]\varphi\de x=\int_{Q}\frac{\hi-\him}{\Jim\tau_N}\varphi\de x,
\end{align}
for every $\varphi\in H^1_{\#}(Q)$ with $\int_{Q}\varphi\de x=0$, where $A$ is defined in \eqref{A}. We now use $\varphi=D_k\eta$ with $\eta\in H^2_{\#}(Q)$, and since
\begin{equation*}
    \frac{D_k\eta}{\Ji}=D_k\left(\frac{\eta}{\Ji}\right)+\frac{\eta D_k\Ji}{\Ji^2},
\end{equation*}
we can integrate by parts to get
\begin{align}\label{interm}
    \int_{Q}\langle ADw, & \hspace{2pt}D\left(\frac{D_k\eta}{\Ji}\right)\rangle \Ji\de x\nonumber\\
    &=\sum_{s,l}\int_Q a_{sl}D_l w D_sD_k\left(\frac{\eta}{\Ji}\right)\Ji\de x+\sum_{s,l}\int_Q a_{sl}D_l w D_s\left(\frac{\eta D_k\Ji}{\Ji^2}\right)\Ji\de x\nonumber\\
    &=-\sum_{s,l}\int_Q D_k(a_{sl}\Ji D_l w) D_s\left(\frac{\eta}{\Ji}\right)\de x-\sum_{s,l}\int_Q D_s(a_{sl}\Ji D_l w)\frac{\eta D_k\Ji}{\Ji^2}\de x\nonumber\\
    &=-\sum_{s,l}\int_Q D_k(a_{sl}\Ji)D_l w D_s\left(\frac{\eta}{\Ji}\right)\de x-\sum_{s,l}\int_Q a_{sl}D_kD_l w D_s\left(\frac{\eta}{\Ji}\right)\Ji\de x\nonumber\\
    &\hspace{10pt}-\sum_{s,l}\int_Q D_s(a_{sl}\Ji)D_l w \frac{\eta D_k\Ji}{\Ji^2}\de x-\sum_{s,l}\int_Q a_{sl} D_sD_l w\frac{\eta D_k\Ji}{\Ji}\de x\nonumber\\
    &=-\int_Q \langle D_k(A\Ji)D w, D\left(\frac{\eta}{\Ji}\right)\rangle\de x-\int_Q \langle AD(D_k w), D\left(\frac{\eta}{\Ji}\right)\rangle\Ji\de x\nonumber\\
    &\hspace{10pt}-\int_Q\langle\div(A\Ji),D w\rangle \frac{\eta D_k\Ji}{\Ji^2}\de x-\int_Q A: D^2 w\frac{\eta D_k\Ji}{\Ji}\de x
\end{align}

Hence, by \eqref{intr2}, \eqref{interm}, and by a density argument, we obtain that 
\begin{align*}
    \int_{Q}\langle AD(D_k w), &\hspace{2pt}D\left(\frac{\eta}{\Ji}\right)\rangle\Ji\de x\\
    &=-\int_Q \langle D_k(A\Ji)Dw, D\left(\frac{\eta}{\Ji}\right)\rangle\de x-\int_Q A:D^2 w \frac{\eta D_k\Ji}{\Ji}\de x\\
    &-\int_Q \langle \div(A\Ji),Dw\rangle\frac{\eta D_k\Ji}{\Ji^2}-\varepsilon\int_{Q}wD_k\eta\left(|B_{i,N}|^2-\frac{1}{p}\Hi^2\right)\de x\nonumber\\
    &+\int_{Q}\left[\div_{\Gin}(D\psi(\nu_{i,N}))+W(E\ui)\right]D_k\eta\de x+\int_{Q}\frac{1}{\Jim}\frac{\partial h_N}{\partial t}D_k\eta\de x,
\end{align*}
for every $\eta\in H^1_{\#}(Q)$. Therefore, by choosing $\eta=D_kw\Ji$ we obtain
\begin{align}\label{stima-k}
    &\int_{Q}\langle AD(D_k w), D(D_k w)\rangle\Ji\de x\nonumber\\
    &=-\int_Q \langle D_k(A\Ji)Dw, D(D_k w)\rangle\de x-\int_Q A:D^2 w D_kw D_k\Ji\de x\nonumber\\
    &-\int_Q \langle \div(A\Ji),Dw\rangle\frac{D_kwD_k\Ji}{\Ji}-\varepsilon\int_{Q}wD_k(D_kw\Ji)\left(|B_{i,N}|^2-\frac{1}{p}\Hi^2\right)\de x\nonumber\\
    &+\int_{Q}\left[\div_{\Gin}(D\psi(\nu_{i,N}))+W(E\ui)\right]D_k(D_kw\Ji)\de x
     \notag\\
  & +\int_{Q}\frac{1}{\Jim}\frac{\partial h_N}{\partial t}D_k(D_kw\Ji)\de x.
\end{align}

Now, by summing \eqref{stima-k} for $k=1,2$, by using Young's inequality, the ellipticity property of matrix $A$, and the estimate of $\div(A\Ji)$ in terms of $D^2\hi$, we conclude that
\begin{align}\label{prim-1}
    \int_Q |D^2 w|^2\de x\leq C\int_Q\left( |Dw|^2|D^2\hi|^2+|\Hi|^{2(p+1)}+|\Hi|^{2(p-1)}|D^2\hi|^4+\left|\frac{\partial h_N}{\partial t}\right|^2+1\right)\de x,
\end{align}
where the constant $C$ depends only on $\Lambda_0$, $D^2\psi$, and on the $C^{1,\alpha}$ bound of \eqref{stima-grad-u}. By Young's inequality, together to \eqref{stimaD2h} and Lemma \ref{A.6}, we have
\begin{align}\label{prim-2}
    \int_Q|\Hi|^{2p-2}|D^2\hi|^4\de x&\leq C\int_Q(|\Hi|^{2p+2}+|D^2\hi|^{2p+2})\de x\leq C\int_Q|\Hi|^{2p+2}\de x\nonumber\\
    &=C\int_Q|w|^{\frac{2(p+1)}{p-1}}\de x\leq C\|D^2w\|^{\frac{p+2}{p}}_{L^{\frac{p}{p-1}}(Q)}\|w\|^{\frac{p^2+p+2}{p(p-1)}}_{L^{\frac{p}{p-1}}(Q)}\nonumber\\
    &\leq \frac{1}{4}\|D^2w\|_{L^{\frac{p}{p-1}}(Q)}^2+C\|w\|^{\frac{p^2+p+2}{(p-1)(p-2)}}_{L^{\frac{p}{p-1}}(Q)}
     \notag\\
  & =\frac{1}{4}\|D^2w\|_{L^{\frac{p}{p-1}}(Q)}^2+C\|\Hi\|^{\frac{p^2+p+2}{p-2}}_{L^{p}(Q)}\nonumber\\
    &\leq \frac{1}{4}\|D^2w\|_{L^{2}(Q)}^2+C,
\end{align}
where in the last two inequalities we used \eqref{Hb} and
\begin{equation}\label{inc}
    L^2\hookrightarrow L^{\frac{p}{p-1}} 
\end{equation}
since $\frac{p}{p-1}<2$, and
\begin{equation}\label{inc2}
     \|w\|_{L^{\frac{p}{p-1}}(Q)}=\|\Hi\|_{L^p(Q)}^{p-1}.
\end{equation}
Moreover, we have
\begin{align}\label{prim-3}
    \int_Q  |Dw|^2|D^2\hi|^2\de x&\leq \|D^2h\|^2_{L^{2(p-1)}(Q)}\|Dw\|^2_{L^{\frac{2(p-1)}{p-2}}(Q)} \notag\\
  & \leq C\|w\|^2_{L^{\frac{2}{p-1}}(Q)}\left(\|D^2w\|^{\frac{p}{2(p-1)}}_{L^2(Q)}\|w\|^{\frac{p-2}{2(p-1)}}_{L^2(Q)}\right)^2\nonumber\\
    & =C\|D^2w\|^{\frac{p}{p-1}}_{L^2(Q)}\|w\|^{\frac{p}{p-1}}_{L^2(Q)} \notag\\
  & \leq C\|D^2w\|^{\frac{p}{p-1}}_{L^2(Q)}\left(\|D^2w\|^{\frac{p-2}{2p}}_{L^{\frac{p}{p-1}}(Q)}\|w\|^{\frac{p+2}{2p}}_{L^{\frac{p}{p-1}}(Q)}\right)^{\frac{p}{p-1}}\nonumber\\
    &\leq C\|D^2w\|^{\frac{3p-2}{2(p-1)}}_{L^2(Q)}\|w\|^{\frac{p+2}{2(p-1)}}_{L^{\frac{p}{p-1}}(Q)}\leq \frac{1}{4}\|D^2w\|^2_{L^2(Q)}+C\|w\|_{L^{\frac{p}{p-1}}(Q)}^{\frac{2(p+2)}{p-2}}\nonumber\\
    &=\frac{1}{4}\|D^2w\|^2_{L^2(Q)}+C\|\Hi\|^{\frac{2(p+2)(p-1)}{p-2}}_{L^p(Q)}\leq \frac{1}{4}\|D^2w\|^2_{L^2(Q)}+C,
\end{align}
where in the first inequality we used Holder's inequality, in the second and third ones we used Lemma \ref{A.6}, in the fourth and the fifth ones we used \eqref{inc} and \eqref{inc2}, and in the last one we used \eqref{Hb}.

Finally, by \eqref{prim-1}, \eqref{prim-2}, and \eqref{prim-3} we get
\begin{equation*}
    \int_Q |D^2w|^2\de x\leq C\int_Q\left(1+\left|\frac{\partial h_N}{\partial t}\right|^2\right)\de x,
\end{equation*}
and by integrating with respect to time and by using \eqref{HuLd} the assertion follows.
\end{proof}

Thanks to the bound \eqref{stimaD2H} provided by the previous theorem, in the next lemma we obtain the convergence of powers of the sum of the squares of principal curvatures $\Htn$ defined in \eqref{Hn}.  

\begin{lemma}\label{lem_conv}
    Let $\Htn$ be the function defined in \eqref{Hn}. Then
    \begin{align}
       |\Htn|^p&\xrightarrow[N\to \infty]{}|H|^p \quad \text{in $L^1(0,T_0;L^1(Q))$},\label{conv-Hp}\\
       |\Htn|^{p-2}\Htn&\xrightarrow[N\to \infty]{}|H|^{p-2}H \quad \text{in $L^1(0,T_0;L^2(Q))$}\label{conv-Hpd},
    \end{align}
    where 
    \begin{equation*}
        H:=-\div\left(\frac{Dh}{\sqrt{1+|Dh|^2}}\right)
    \end{equation*}
and $h$ is the function provided by Theorem \ref{thm-t0}.
\end{lemma}
\begin{proof}
   The proof is analogous of the one in \cite[Corollary 3.15]{FFLM3} based on \cite[Lemma 3.13]{FFLM3}.
\end{proof}

Finally, we can prove the short time existence for \eqref{syst}.

\begin{proof}[Proof of Theorem~\ref{teoconvfin}]
    Fix $t\in(0,T_0)$ and let $\{i_k\}_k$ and $\{N_k\}_k$ be sequences such that $t_k:=i_k\tau_{N_k}\to t$ for $k\to +\infty$. Now, by summing \eqref{EL2} for $i=1,...,i_k$ we obtain
   \begin{align}\label{eqk}
    \int_0^{t_k}\int_Q &\Wt\varphi\de x\de t+\int_0^{t_k}\int_Q\langle D\psi(-D\hti,1),(-D\varphi,0)\rangle \de x\de t\nonumber\\
   & -\varepsilon\int_0^{t_k}\int_Q|\Ht|^{p-2}\Ht\left[\Delta \varphi-\frac{\langle D^2\varphi D\hti,D\hti \rangle}{\Ji^2}-\frac{\Delta\hti \langle D\hti,D\varphi \rangle}{\Jt^2}\right]\de x\de t\nonumber\\
   &-\varepsilon\int_0^{t_k}\int_Q|\Ht|^{p-2}\Ht\Bigg[3\frac{\langle D\hti,D\varphi\rangle\langle D^2\hti D\hti,D\hti\rangle}{\Jt^4} \notag\\
  & \qquad\qquad\qquad\qquad\qquad\qquad\qquad\qquad\qquad\qquad\qquad-2\frac{\langle D^2\hti D\hti,D\varphi \rangle}{\Jt^2}\Bigg]\de x\de t\nonumber\\
   &+\frac{\varepsilon}{p}\int_0^{t_k}\int_Q |\Ht|^p\frac{\langle D\hti,D\varphi\rangle}{\Ji}\de x\de t \notag\\
  & +\int_0^{t_k}\int_Q\frac{1}{\Jt(\cdot,\cdot-\tau_{N_k})}\frac{\partial h_{N_k}}{\partial t}\varphi\de x\de t=0,
\end{align} 
where $h_{N_k}$, $\tilde{h}_{N_k}$, and $\Ht$ are defined in \eqref{linear_h}, \eqref{constant_h}, and \eqref{Hn}, and we define 
\begin{align*}
    &\Wt(x,t):=W(Eu_{i,N_k}(x,h_{i,N_k}(x))) &&\text{for every $(x,t)\in \R^2\times [(i-1)\tau_{N_k},i\tau_{N_k})$}\\
    &\Jt(x,t):=\sqrt{1+|D\hti|^2}, &&\text{for every $(x,t)\in \R^2\times [0,T_0]$}.
\end{align*}
Now, we claim that
\begin{equation}\label{conv-V}
\frac{1}{\Jt(\cdot,\cdot-\tau_{N_k})}\frac{\partial h_{N_k}}{\partial t}\xrightharpoonup[k\to\infty]{}\frac{1}{J}\frac{\partial h}{\partial t}\quad \text{in $L^2(0,T_0;L^2(Q))$}.
\end{equation}
To prove the claim, it is enough to notice that  
\begin{align*}
    &\left|\int_0^{T_0}\int_Q\left(\frac{1}{\Jt(\cdot,\cdot-\tau_{N_k})}\frac{\partial h_{N_k}}{\partial t}-\frac{1}{J}\frac{\partial h}{\partial t}\right)\eta\de x\de t\right|\\
    &\leq \left|\int_0^{T_0}\int_Q\left(\frac{1}{\Jt(\cdot,\cdot-\tau_{N_k})}-\frac{1}{J}\right)\frac{\partial h_{N_k}}{\partial t}\eta\de x\de t\right|+\left|\int_0^{T_0}\int_Q\left(\frac{\partial h_{N_k}}{\partial t}-\frac{\partial h}{\partial t}\right)\frac{\eta}{J}\de x\de t\right|\nonumber\\
    &\leq \left\|\frac{\partial h_{N_k}}{\partial t}\right\|_{L^2(0,T_0,L^2(Q))} \left\|\frac{\eta}{\Jt(\cdot,\cdot-\tau_{N_k})}-\frac{\eta}{J}\right\|_{L^2(0,T_0,L^2(Q))}+\left|\int_0^{T_0}\int_Q\left(\frac{\partial h_{N_k}}{\partial t}-\frac{\partial h}{\partial t}\right)\frac{\eta}{J}\de x\de t\right|,
\end{align*}
for every $\eta\in L^2(0,T_0,L^2(Q))$, and hence, from \eqref{HuLd}, \eqref{convhu}, and \eqref{conv-holder} it follows \eqref{conv-V}.

Moreover, in view of Theorems \ref{teo_h_conv2} and \ref{thm-t0} we get
\begin{align}
        \int_0^{t_k}\int_Q \Wt\varphi\de x\de t&\xrightarrow[k\to\infty]{}\int_0^{t}\int_Q W(E(u(x,h(x,s),s))\varphi\de x\de s,\label{thm-conv1}\\
       \int_0^{t_k}\int_Q\langle D\psi(-D\hti,1),(-D\varphi,0)\rangle \de x\de t&\xrightarrow[k\to\infty]{}\int_0^{t}\int_Q\langle D\psi(-Dh,1),(-D\varphi,0)\rangle \de x\de s,\label{thm-conv2}
    \end{align}
and by using again Theorem \ref{teo_h_conv2}, \eqref{conv-Hp}, and \eqref{conv-Hpd}, we obtain
\begin{align}
    &\int_0^{t_k}\int_Q |\Ht|^p\frac{\langle D\hti,D\varphi\rangle}{\Ji}\de x\de t\xrightarrow[k\to\infty]{}\int_0^{t}\int_Q |H|^p\frac{\langle Dh,D\varphi\rangle}{J}\de x\de s,\label{thm-conv3}\\
    &\int_0^{t_k}\int_Q|\Ht|^{p-2}\Ht\left[\Delta \varphi-\frac{\langle D^2\varphi D\hti,D\hti \rangle}{\Ji^2}\right]\de x\de t\nonumber\\
    &\hspace{5cm}\xrightarrow[k\to\infty]{}\int_0^{t}\int_Q|H|^{p-2}H\left[\Delta \varphi-\frac{\langle D^2\varphi Dh,Dh \rangle}{J^2}\right]\de x\de s\label{thm-conv4}.
\end{align}

In order to establish the convergence of the other terms, we firstly observe that by \eqref{wdp} and \eqref{conv-holder} we have $\Delta \hti(\cdot, t)\rightharpoonup \Delta h(\cdot,t)$ for every $t\in (0,T_0)$. Moreover, by \eqref{conv-Hpd} we have for a.e.\ $t\in(0,T_0)$ we have that $(|\Ht|^{p-2}\Ht)(\cdot,t)\rightarrow (|H|^{p-2}H)(\cdot,t)$ in $L^2(Q)$, and hence
\begin{equation}\label{prim}
    \int_Q |\Ht|^{p-2}\Ht\frac{\Delta\hti \langle D\hti,D\varphi \rangle}{\Jt^2}\de x\xrightarrow[k\to\infty]{}\int_Q |H|^{p-2}H\frac{\Delta h \langle Dh,D\varphi \rangle}{
    J^2}\de x.
\end{equation}
Since by \eqref{linear_h}, \eqref{wdp}, and \eqref{conv-Hpd} we have
\begin{align}\label{lk}
    \left|\int_Q|\Ht|^{p-2}\Ht\frac{\Delta\hti \langle D\hti,D\varphi \rangle}{\Jt^2}\de x\right|&\leq C\|\Delta \hti\|_{L^2(Q)}\||\Ht|^{p-2}\Ht\|_{L^2(Q)}\nonumber\\
    &\leq C\||\Ht|^{p-2}\Ht\|_{L^2(Q)},
\end{align}
\PPP and 
\begin{equation}\label{lk2}
C\||\Ht|^{p-2}\Ht\|_{L^2(Q)}\xrightarrow[k\to \infty]{L^1(0,T_0)}\||H|^{p-2}H\|_{L^2(Q)}, 
\end{equation}
\OOO
the generalized Lebesgue dominated convergence theorem, together with \eqref{lk}\PPP, \eqref{lk2}, \OOO and \eqref{prim}, implies that
\begin{equation}\label{thm-conv5}
      \int_0^{t_k}\int_Q |\Ht|^{p-2}\Ht\frac{\Delta\hti \langle D\hti,D\varphi \rangle}{\Jt^2}\de x\de t\xrightarrow[k\to\infty]{}\int_0^t\int_Q |H|^{p-2}H\frac{\Delta h \langle Dh,D\varphi \rangle}{
    J^2}\de x\de s.
\end{equation}
For the remaining terms, by arguing in analogously we obtain
\begin{equation}\label{thm-conv6}
    \int_0^{t_k}\int_Q |\Ht|^{p-2}\Ht\frac{\langle D^2\hti D\hti,D\varphi \rangle}{\Jt^2}\de x\de t\xrightarrow[k\to\infty]{}\int_0^t\int_Q |H|^{p-2}H\frac{\langle D^2 h Dh,D\varphi \rangle}{J^2}\de x\de s,
\end{equation}
and
\begin{align}\label{thm-conv7}
      \int_0^{t_k}\int_Q |\Ht|^{p-2}\Ht&\frac{\langle D\hti,D\varphi\rangle\langle D^2\hti D\hti,D\hti\rangle}{\Jt^4}\de x\de t\nonumber\\
      &\hspace{3cm}\xrightarrow[k\to\infty]{}\int_0^t\int_Q |H|^{p-2}H\frac{\langle Dh,D\varphi\rangle\langle D^2 h Dh,Dh\rangle}{J^4}\de x\de s.
\end{align}

Finally, by \eqref{conv-V}, \eqref{thm-conv1}, \eqref{thm-conv2}, \eqref{thm-conv3}, \eqref{thm-conv4}, \eqref{thm-conv5}, \eqref{thm-conv6}, and \eqref{thm-conv7}, we pass to the limit in \eqref{eqk}, obtaining that the limiting function $h$ satisfies
\begin{align}\label{final}
    \int_0^{t}\int_Q\frac{1}{J}\frac{\partial h}{\partial t}\varphi\de x\de s&=-\int_0^{t}\int_Q W(E(u(x,h(x,s),s))\varphi\de x\de s \notag\\
  & -\int_0^{t}\int_Q\langle D\psi(-Dh,1),(-D\varphi,0)\rangle \de x\de s\nonumber\\
   & +\varepsilon\int_0^{t}\int_Q|H|^{p-2}H\Bigg[\Delta \varphi-\frac{\langle D^2\varphi Dh,Dh \rangle}{J^2}-\frac{\Delta h \langle Dh,D\varphi \rangle}{J^2}- \notag\\
  & \qquad\qquad\qquad\qquad\qquad\qquad\qquad\qquad\qquad 2\frac{\langle D^2 h Dh,D\varphi \rangle}{J^2}\Bigg]\de x\de s\nonumber\\
   &+\varepsilon\int_0^{t}\int_Q\Bigg[3|H|^{p-2}H\frac{\langle Dh,D\varphi\rangle\langle D^2 h Dh,Dh\rangle}{J^4} \notag\\
  & \qquad\qquad\qquad\qquad\qquad\qquad\qquad\qquad-\frac{1}{p}|H|^p\frac{\langle Dh,D\varphi\rangle}{J}\Bigg]\de x\de s.
\end{align} 
Now, by letting $\varphi$ vary in a countable dense subset of $H^1_{\#}(Q)$ and by differentiating \eqref{final} with respect to t, we obtain
\begin{align}\label{last}
    &-\int_Q W(E(u(x,h(x,t),t))\varphi\de x-\int_Q\langle D\psi(-Dh,1),(-D\varphi,0)\rangle \de x-\frac{\varepsilon}{p}\int_Q |H|^p\frac{\langle Dh,D\varphi\rangle}{J}\de x\nonumber\\
   & +\varepsilon\int_Q|H|^{p-2}H\left[\Delta \varphi-\frac{\langle D^2\varphi Dh,Dh \rangle}{J^2}-\frac{\Delta h \langle Dh,D\varphi \rangle}{J^2}-2\frac{\langle D^2 h Dh,D\varphi \rangle}{J^2}\right]\de x\nonumber\\
   &+3\varepsilon\int_Q|H|^{p-2}H\frac{\langle Dh,D\varphi\rangle\langle D^2 h Dh,Dh\rangle}{J^4}\de x=\int_Q\frac{1}{J}\frac{\partial h}{\partial t}\varphi\de x.
\end{align} 
Since by \eqref{stimaD2H} $|H|^{p-2}H\in L^2(0,T_0;H^2_{\#}(Q))$, by arguing as in Theorem \ref{boundD2H} we have that \eqref{last} is equivalent to  
\begin{align*}
    -\varepsilon\int_{\Gamma_h}D_{\Gamma_h}(|H|^{p-2}H)D_{\Gamma_h}\phi\de \H^2&+\varepsilon\int_{\Gamma_h}|H|^{p-2}H\left(|B|^2-\frac{1}{p}H^2\right)\phi \de \H^2\nonumber\\
    &-\int_{\Gamma_h}\left[\div_{\Gamma_h}(D\psi(\nu))+W(Eu)\right]\phi\de \H^2=\int_{\Gamma_h}V\phi\de \H^2
\end{align*}
for every $t\in(0,T_0)$, with $\phi:=\frac{\varphi}{J}$.

Now, we want to show that the energy decreases during the evolution. Thanks to \eqref{en-dec} we can say that the function $t\mapsto \F(\htin(\cdot,t),\utn(\cdot,t))$ is nonincreasing, where $\htin(\cdot,t)$ and $\utn(\cdot,t)$ are defined in \eqref{constant_h} and \eqref{constant_u}. Moreover, by using \eqref{conv-Hp} we have, up to a subsequence, that for a.e.\ $t\in(0,T_0)$ it holds $\Htn(\cdot,t)\rightarrow H(\cdot,t)$ in $L^p(Q)$. Thanks to this, we can use \eqref{conv-inf} and \eqref{conv-trac-tu} to get $\F(\htin(\cdot,t),\utn(\cdot,t))\to \F(h(\cdot,t),u(\cdot,t))$ as $N\to \infty$ for all such $t$, which implies \eqref{ugF}. Now, to prove \eqref{disF} we consider $t\in Z_0$ and $t_N\to t+$ with $t_N\notin Z_0$ as $N\to \infty$ for every $N\in \N$. Since $h(\cdot,t_N)\rightharpoonup h(\cdot,t)$ in $\Ws$ by \eqref{wdp}, by using the lower semicontinuity of $\F$ we obtain
\begin{equation*}
    \F(h(\cdot,t),u(\cdot,t))\leq \liminf_{N\to \infty}\F(h(\cdot,t_N),u(\cdot,t_N))=\lim_{N\to\infty}g(t_N)=g(t+),
\end{equation*}
which concludes the proof.
\end{proof}

\bigskip\bigskip

\section{Appendix}\label{sec7}

For the proof of the following lemmas see \cite{FFLM3}.

\begin{lemma}\label{A.1}
Let $M<0$ and $c_0>0$. Let $h_1,h_2\in C^{1,\alpha}_{\#}(Q)$ for some $\alpha\in(0,1)$, with $\|h_i\|_{C^{1,\alpha}_{\#}(Q)}\leq M$ and $h_i\geq c_0$ for $i=1,2$, and let $u_1$ and $u_2$ be the corresponding elastic equilibria of $h_1$ and of $h_2$, respectively. Then
\begin{equation*}
    \|E(u_1(\cdot,h_1(\cdot)))-E(u_2(\cdot,h_2(\cdot)))\|_{C^{1,\alpha}_{\#}(Q)}\leq C\|h_1-h_2\|_{C^{1,\alpha}_{\#}(Q)},
\end{equation*}
for some constant $C>0$ depending only on $M$, $c_0$, and $\alpha$.
\end{lemma}

\begin{lemma}\label{A.2}
    Let $p>2$, $u\in L^{\frac{p}{p-1}}(Q)$ such that
    \begin{equation*}
       \int_Q u A:D^2\varphi\de x+\int_Q\langle b,D\varphi\rangle\de x+\int_Q c\varphi\de x=0 
    \end{equation*}
for every $\varphi\in C^{\infty}_{\#}(Q)$ with $\int_Q \varphi\de x=0$, where $A\in W^{1,p}_{\#}(Q;\R^{2\times 2}_{sym})$ satisfies standard uniform elliptic condition \eqref{suec}, $b\in L^1(Q;\R^2)$ and $c\in L^1(Q)$. Then $u\in L^q(Q)$ for every $q\in (1,2)$. Moreover, if $b,u\div A\in L^r(Q;\R^2)$ and $c\in L^r(Q)$ for some $r>1$, then $u\in W^{1,r}_{\#}(Q)$.
\end{lemma}

In the next lemma we denote by $Lu$ an elliptic operator of the form
\begin{equation}\label{op-L}
    Lu:=\sum_{i,j}a_{ij}(x)D_{ij}u+\sum_{i}b_i(x)D_iu
\end{equation}
where all the coefficients are $Q$-periodic functions, the $a_{ij}$ are continuous, and the $b_i$ are bounded. Moreover, there exist $\lambda,\Lambda>0$ such that
\begin{equation}\label{suec}
    \Lambda|\xi|^2\geq \sum_{i,j}a_{ij}\xi_i\xi_j\geq \lambda|\xi|^2
\end{equation}
for every $\xi\in\R^2$, $\sum_{i}|b_i|\leq \Lambda$.

\begin{lemma}\label{A.3}
Let $p\geq 2$. Then, there exists $C > 0$ such that for all $u\in  \Ws$ we have
\begin{equation*}
    \|D^2 u\|_{\LpQ}\leq C\|Lu\|_{\LpQ},
\end{equation*}
where $L$ is the differential operator defined in \eqref{op-L}. The constant $C$ depends only on $p$, $\lambda$, $\Lambda$, and the moduli of continuity of the coefficients $a_{ij}$.
\end{lemma}

\begin{lemma}\label{A.6} 
Let $\Omega\subset \R^n$ be a bounded open set satisfying the cone condition. Let $s$, $j$ , and $m$ be integers such that $0\leq s\leq j\leq m$. Let $1\leq p\leq q < \infty$ if $(m-j)p\geq n$, and let $1\leq p\leq q \leq \infty$ if $(m-j)p>n$. Then, there exists $C > 0$ such that
\begin{equation}\label{interp}
    \|D^j f\|_{\LqQ}\leq C\left(\|D^m f\|_{\LpQ}^{\theta}\|D^s f\|_{\LpQ}^{1-\theta}+\|D^s f\|_{\LpQ}\right)
\end{equation}
for every $f\in W^{m,p}(\Omega)$, where
\begin{equation*}
    \theta:=\frac{1}{m-s}\left(\frac{n}{p}-\frac{n}{q}+j-s\right).
\end{equation*}
Moreover, if $\Omega$ is a cube, $f\in W^{m,p}_{\#}(Q)$ and, if either $f$ vanishes at the boundary or $\int_{\Omega}f \de x=0$, then \eqref{interp} holds in the stronger form
\begin{equation*}
    \|D^j f\|_{\LqQ}\leq C\|D^m f\|_{\LpQ}^{\theta}\|D^s f\|_{\LpQ}^{1-\theta}.
\end{equation*}
\end{lemma}

\bigskip\bigskip

{\begin{acknowledgements}
The authors acknowledge the financial support received from the Austrian Science Fund (FWF) through the project TAI 293. Furthermore, the authors are members of the Italian ``Gruppo Nazionale per l'Analisi Matematica, la Probabilit\`a e le loro Applicazioni'' (GNAMPA-INdAM) and have received funding from the GNAMPA 2022 project CUP: E55F22000270001 and GNAMPA 2023 project CUP: E53C22001930001. Moreover,  P. Piovano thanks for the funding received from the Austrian Science Fund (FWF) through the project P 29681 and from BMBWF through the OeAD-WTZ project HR 08/2020, and is grateful for the support received as \emph{Visitor in the Theoretical Visiting Sciences Program} (TSVP) at the Okinawa Institute of Science and Technology (OIST), Japan. Finally, P. Piovano  acknowledges the support obtained by the Italian Ministry of University and Research  (MUR) through  the PRIN Project ``Partial differential equations and related geometric-functional inequalities''. 
\end{acknowledgements}}

\end{document}